\documentclass[11pt]{amsart}
\usepackage{amsmath,amscd,amssymb,amsfonts,amsthm}
\usepackage{tikz-cd} 
\usepackage{tikz}

\usepackage[normalem]{ulem} 

\usepackage[color,cmtip,matrix,arrow]{xy}

\newtheorem{theorem}{Theorem}[section]

\newtheorem{proposition}[theorem]{Proposition}

\newtheorem{lemma}[theorem]{Lemma}
\theoremstyle{definition}
\newtheorem{definition}[theorem]{Definition}
\theoremstyle{remark}
\newtheorem{assumption}[theorem]{Assumption}
\newtheorem{remark}[theorem]{Remark}

\newtheorem{example}[theorem]{Example}

\newcommand\I{\mathcal{I}}

\newcommand{\ca}{\mathcal}

\newcommand{\R}{\mathbb{R}}

\newcommand{\on}{\operatorname}

\renewcommand{\subset}{\subseteq}

\renewcommand{\ker}{ \on{ker}}
\newcommand{\im}{ \on{im}}

\newcommand\qu{/\kern-.7ex/} 

\newcommand{\del}{\delta}

\newcommand{\D}{\ca{D}}

\begin{document}
	\sloppy
	\title{Poisson cohomology of ``book'' Poisson structures}
	
	\author{Henrique Bursztyn }
 \address{IMPA, Estrada Dona Castorina 110\\ Rio de Janeiro, 22460-320, Brasil}
 \email{henrique@impa.br}

 \author{Hudson Lima}
 \address{UFAM, Av. General Rodrigo Octavio Jordão Ramos, 1200 \\
 Manaus, AM, 69067-005, Brasil}
 \email{hudson04br@gmail.com}

\begin{abstract} We compute the Poisson cohomology of the linear Poisson structure  dual to the $n$-dimensional Lie algebra defined by
$$
[e_0,e_i]=e_i, \quad [e_i,e_j]=0, \qquad i,j=1,\ldots,n-1.
$$

\end{abstract}
 
	\maketitle
	\tableofcontents
	
	\section{Introduction}

Poisson cohomology, introduced in  \cite{Lich}, is a fundamental invariant of Poisson manifolds.
For a Poisson manifold $(M,\Lambda)$, it is defined as the cohomology of the complex
$$
d_\Lambda=[\Lambda,\cdot]: \mathfrak{X}^\bullet(M)\to \mathfrak{X}^{\bullet+1}(M),
$$ 
where $[\cdot,\cdot]$ is the Schouten bracket on multivector fields $\mathfrak{X}^\bullet(M)$. 

Poisson cohomology codifies important information about the Poisson structure: For example,  in degree zero, it is the space of Casimir functions (i.e., functions that are constant along symplectic leaves); in degree one, it identifies with the Lie algebra of infinitesimal outer Poisson automorphisms (that is, the space of Poisson vector fields modulo hamiltonian vector fields); in degree two, it describes the space of infinitesimal deformations of the Poisson structure.

Explicit computations of Poisson cohomology are notoriously difficult; known results and calculation methods can be found e.g. in \cite{CFMbook,DZ}. In the particular context of {\em linear} Poisson structures, the case of duals of compact Lie algebras is treated in \cite{GinzWein}, while \cite{Zeiser,MZ} describe the Poisson cohomology of all 3-dimensional examples.

The main goal of this paper is to fully compute, on appropriate open subsets $U\subseteq \R^n$,  the Poisson cohomology of the linear Poisson structure   dual to the Lie algebra defined by the following relations: for $n\geq 2$,
$$
[e_0,e_j]=e_j, \quad [e_i,e_j]=0, \qquad i,j=1,\ldots,n-1.
$$
For $n=3$, this is often referred to as the ``book'' Lie algebra; in this case the corresponding Poisson cohomology in $\R^3$ is described in \cite[Sec.~7]{Zeiser} with a different method. Keeping the terminology for any $n\geq 2$, we refer to the dual linear Poisson structure as the {\em book} Poisson structure
on $\R^{n}$: in coordinates $(t,u_1, ..., u_{n-1})$, it is given by
\begin{equation}\label{eq:book}
    \Lambda_{book} = \partial_t \wedge \big( u_1\partial_{u_1}+ u_2\partial_{u_2} + \cdots + u_{n-1}\partial_{u_{n-1}} \big).
    \end{equation}

The terminology ``book'' is justified by the symplectic leaves of $\Lambda_{book}$: its 0-dimensional leaves are
points along the $t$-axis (the ``binding''); away from the $t$-axis, $\Lambda_{book}$ is regular with 2-dimensional leaves (the ``pages'') given by the fibers of the map
$$
\mathbb{R}^n\setminus \{(t,0,\ldots,0),\; t\in \mathbb{R}\} \to \mathbb{S}^{n-2}, \quad 
(t, u_1,\ldots,u_{n-1}) \mapsto \frac{1}{|u|}(u_1,\ldots,u_{n-1}),
$$
i.e., to each $c=(c_1,...,c_{n-1})\in\mathbb{S}^{n-2}$, there corresponds a leaf given by 
\begin{equation}\label{eq:bookleaf}
\{ (t, \lambda c),\; t\in\R, \ \lambda\in\R^+ \}.
\end{equation}
\begin{figure}[h!]
    \centering
    \includegraphics[width=0.50\linewidth]{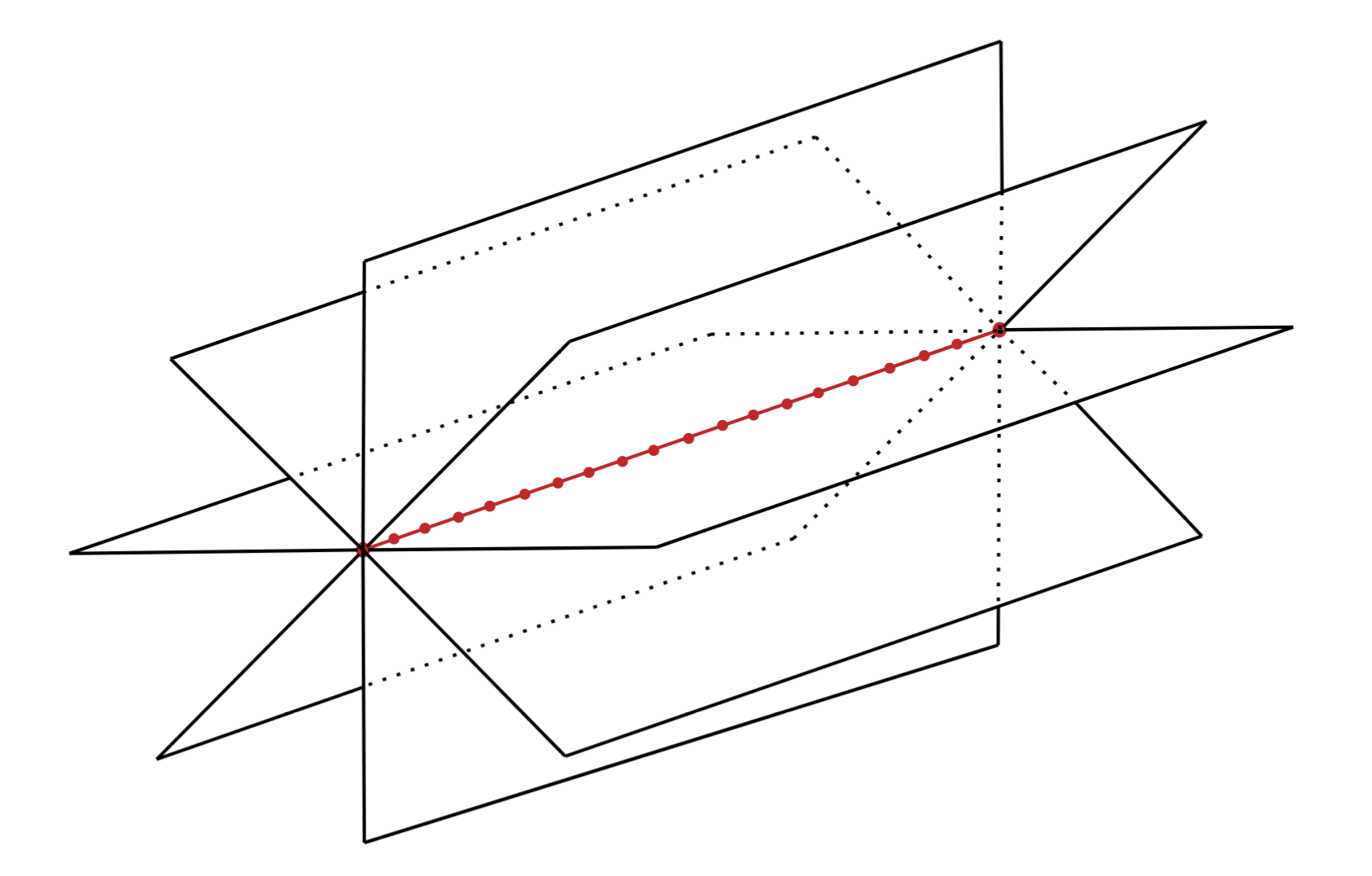}
    \caption{Leaves of $\Lambda_{book}$}
    \label{fig}
\end{figure}

Besides its intrinsic interest, our computation of the Poisson cohomology
of book Poisson structures on suitable open subsets
$U\subseteq\mathbb{R}^n$ is central to the study of certain
Poisson structures on spheres $\mathbb{S}^n$ in \cite{BL}. A linearization
result in \cite{BL} yields a cover of the $n$-sphere 
by two open sets, each Poisson diffeomorphic to an open subset of
$\mathbb{R}^n$ equipped with the book Poisson structure. The results in this paper
provide the local cohomological input for computing the Poisson
cohomology of the sphere via the associated Mayer--Vietoris sequence.
The need to handle both the covering sets and their overlap explains
why we work with suitable open subsets of $\mathbb{R}^n$, rather than
only with $\mathbb{R}^n$ itself. The motivation for studying these
structures on $\mathbb{S}^n$ comes from the theory of Poisson Lie groups: for $n=3$,
the Poisson structure under consideration is the standard
(Lu--Weinstein) Poisson structure on $\mathbb{S}^3=\mathrm{SU}(2)$ \cite{LuWe}.

The paper is structured as follows:
We state the main results describing the Poisson cohomology of $\Lambda_{book}$  in  $\S$ \ref{sec:mainresults}, and their proofs are presented in $\S$ \ref{sec:proofs}.

\bigskip

\noindent{\bf Acknowledgments.} 
We thank Ioan Marcut and Florian Zeiser for  valuable comments and helpful advice.
We  are also thankful to Rui Fernandes and Marco Zambon for stimulating discussions, and the anonymous referee for several useful comments. This project had the financial support of CNPq (Brazilian National Research Council). H. L.  thanks IMPA for its hospitality during several stages of this work.

\bigskip

\noindent{\bf Notation.}
We consider $\R^{n}=\R \times \R^{n-1}$ with coordinates $(t,u_1, ..., u_{n-1})$.
For an open subset $U\subset\R^n$, let
\begin{equation}\label{eq:Wdt}
\mathcal{W}^k(U) = \{\mu\in\mathfrak{X}^k(U) \ : \ \iota_{dt}\mu =0 \},
\end{equation}
so that any $\mu\in \mathfrak{X}^k(U)$ can be uniquely written in the form
\begin{equation}\label{eq:mudecomp}
\mu= \partial_t\wedge a + b,
\end{equation}
for $a\in \mathcal{W}^{k-1}(U)$ and $b\in \mathcal{W}^k(U)$.

For a vector field $X\in\mathfrak{X}^1(U)$ and $\mu\in\mathcal{W}^k(U)$, we denote by $X(\mu)\in \mathcal{W}^k(U)$ the multivector field obtained by applying $X$ to the coordinate functions of $\mu= \sum_{i_1<\ldots <i_k}\mu_{i_1\ldots i_k}\partial_{u_{i_1}}\wedge\ldots\wedge\partial_{u_{i_k}}$,
$$
X(\mu)=\sum_{i_1<\ldots <i_k} X(\mu_{i_1\ldots i_k})\partial_{u_{i_1}}\wedge\ldots\wedge\partial_{u_{i_k}}.
$$

Given a multi-index $I=(i_1,i_2,\ldots,i_k)$ where $1\leq i_j\leq n-1$ for all $j=1,\ldots,k$, we will write $|I|=k$,
$$
u_I=u_{i_1} u_{i_2}\cdots u_{i_k} \ \ \mbox{ and } \ \ \partial_I=\partial_{u_{i_1}}\wedge \partial_{u_{i_2}}\wedge\cdots\wedge\partial_{u_{i_k}}.
$$
We will assume that $i_1\leq i_2\leq\cdots\leq i_k$ for $u_I$, and the strict inequalities for $\partial_I$. With some abuse of notation, we adopt the following conventions: 
\begin{itemize}
\item Given a multi-index $I$, we shall keep the same notation $I$ for its set of elements.
    \item Multi-indexes of the form $I=(i)$, with length $1$, will be denoted by $i$. 
\item We will
write $j\in I$ to denote that $j\in \{i_1,i_2,\ldots,i_k\}$.
\item If $i$ is an element of $I$ we will use 
    $$
        I\setminus i
    $$
    to denote the multi-index obtained from $I$ by lowering the multiplicity of $i$ by one, so 
    $$
        u_I = u_i u_{(I\setminus i)}.
    $$ 
    
\end{itemize}

Throughout the paper, for a real vector subspace $V\subseteq \mathfrak{X}^1(U)$, the notation $\wedge^\bullet V$ denotes the graded subalgebra of multivector fields generated by $V$ (rather than its abstract exterior algebra).

\section{Book Poisson cohomology} \label{sec:mainresults}
Consider $\R^n=\R\times \R^{n-1}$ with coordinates $(t,u_1,\ldots,u_{n-1})$.
We write the book Poisson structure on $\mathbb{R}^n$ as
$$ 
\Lambda_{book} = \partial_t\wedge E,
$$
where we use the notation
$$
E= u_1\partial_{u_1}+ u_2\partial_{u_2} + \cdots + u_{n-1}\partial_{u_{n-1}}
$$ 
for the Euler vector field on $\R^{n-1}$, regarded as a vector field on $\R^n=\R\times \R^{n-1}$.

We denote by $d_{book}$ the Poisson differential of $\Lambda_{book}$, and by  $H^\bullet_{book}(U)$ the Poisson cohomology on an open subset $U\subseteq \R^n$.


\begin{lemma}\label{lem:PDE_cocyle}

Let $\mu = \partial_t\wedge a + b \in \mathfrak{X}^k(U)$, as in \eqref{eq:mudecomp}. Then
$$
d_{book} (\mu) = \partial_t\wedge (E\wedge \partial_t a + E(b)-kb) - E\wedge \partial_t b
$$
\end{lemma}

\begin{proof}
    Since $\mathcal{L}_E\partial_j=-\partial_j$, it follows that $\mathcal{L}_E\partial_J = -k\partial_J$ for $k=|J|$, and 
    $$
    \mathcal{L}_E(b_J\partial_J)=E(b_J)\partial_J +b_J\mathcal{L}_E\partial_J = (E(b_J)-kb_J)\partial_J,
    $$
    for every $b_J\in C^\infty(U)$. Therefore, for $b\in\mathcal{W}^k(U)$, we have $\mathcal{L}_E(b) = E(b)-kb$. Also, since $\partial_t$ is a coordinate vector field, for $\mu = \partial_t\wedge a + b$ we have that $\mathcal{L}_{\partial_t}\mu = \partial_t\wedge \partial_ta + \partial_tb$.

Recall that the Schouten bracket satisfies
    \begin{equation*}
        [X\wedge Y, \mu] = X\wedge\mathcal{L}_Y\mu - Y\wedge\mathcal{L}_X\mu,
    \end{equation*}
    for $X,Y\in\mathfrak{X}(U)$.
Hence
    \begin{eqnarray*}
        d^k_{book}(\mu) &=& [\partial_t\wedge E, \ \mu] = \partial_t\wedge\mathcal{L}_E\mu - E\wedge\mathcal{L}_{\partial_t}\mu \\
        &=& \partial_t\wedge\Big(\partial_t\wedge\mathcal{L}_E(a) + \mathcal{L}_E(b)\Big) - E\wedge\Big( \partial_t\wedge\partial_ta + \partial_tb\Big) \\
        &=& \partial_t\wedge\Big( E\wedge\partial_ta + E(b)-kb \Big) - E\wedge\partial_tb.
    \end{eqnarray*}
    
\end{proof}

Consider the subspace 
$$
\mathrm{span}_\R\{\partial_t, u_i\partial_{u_j}\} \subset \mathfrak{X}^1(U),
$$
which has dimension $(n-1)^2+1$ and contains $E$, and consider the graded subalgebra 
\begin{equation}\label{eq:subalg}
\wedge^\bullet \mathrm{span}_\R\{\partial_t, u_i\partial_{u_j}\}\subseteq \mathfrak{X}^\bullet(U).
\end{equation}
All elements $\partial_t, u_i\partial_{u_j}$ are $d_{book}$-closed (see Lemma~\ref{lem:PDE_cocyle}), so we have a natural ring homomorphism
\begin{equation}\label{eq:naturalhom}
\wedge^\bullet \mathrm{span}_\R\{\partial_t, u_i\partial_{u_j}\} \longrightarrow
H^\bullet_{book}(U).
\end{equation}

We will consider open subsets $U \subseteq \R^n=\R\times\R^{n-1}$ with the following property:

\begin{assumption}\label{assump:U}
The projection $p\colon U\to\R^{n-1}$, $(t,u)\mapsto u$ has connected fibers and its image $p(U)$ is star-shaped with respect to $0\in\R^{n-1}$. 
\end{assumption}


\begin{theorem}\label{thm:main1}
Let $U\subset \R^n$ be an open subset as in Assumption \ref{assump:U}, and consider the subalgebra \eqref{eq:subalg}. Then the map \eqref{eq:naturalhom} is onto and its kernel is the ideal generated by $E$, so it induces a ring isomorphism
$$
\frac{\wedge^\bullet \mathrm{span}_\R\{\partial_t, u_i\partial_{u_j}\}}{\langle E \rangle} \stackrel{\sim}{\longrightarrow} H^\bullet_{book}(U).
$$
\end{theorem}

This result is a consequence of Theorem~\ref{thm:poissoncohU} below.

Note that one can take $U=\R^n$ in the previous theorem, but considering more general open subsets will be important for the applications in \cite{BL}.

\begin{remark}[On the domain $U$]
The theorem implies that the book Poisson cohomology on $U$ is
{\em finite-dimensional} in every degree and for arbitrary $n$, a property that heavily relies 
on the hypotheses on
$U$. This is already visible
in degree zero: under Assumption~\ref{assump:U}, the two-dimensional
leaves have common accumulation points on the ``binding'' ($t$-axis), so continuity
forces every Casimir function to be constant.

The situation can be very different in general. For example, the complement
$V=\mathbb{R}\times(\mathbb{R}^{n-1}\setminus\{0\})$ of the binding
is identified with
$\mathbb{R}^2\times \mathbb{S}^{n-2}$ equipped with
$\partial_t\wedge\partial_s$ via $(t,s,\theta)\mapsto (t,e^s\theta)$. Then the product formula for Poisson
cohomology \cite{Vaisman,Xu} gives
\[
H^k_{{book}}(V)\cong\mathfrak{X}^k(\mathbb{S}^{n-2}),
\]
which is infinite-dimensional for $0\leq k\leq n-2$ when $n\geq3$.
 \hfill $\diamond$
\end{remark}

The following special cases can be directly deduced from the theorem.

\begin{example}
For $n=2$, 
$$
H^0_{book}(U) = \mathrm{span}_\R\{[1]\}\cong \R , \quad H^1_{book}(U)=\mathrm{span}_\R\{[\partial_t]\}\cong \R, \quad 
H^2_{book}(U)=0.
$$
When $n=2$, $\Lambda_{book}$ is a {\em $b$-Poisson} (a.k.a. {\em log symplectic}) structure, with singular locus given by the $t$-axis. In this case the above computation can be also obtained from the general description of Poisson cohomology for $b$-Poisson structures found  in \cite{GMP,mot}. \hfill $\diamond$
\end{example}

\begin{example}
 For $n=3$, we obtain 
\begin{eqnarray*}
H_{book}^0(U) &\cong & \R,\\
H_{book}^1(U) & \cong & \R^4, \ \mbox{ with basis } \{ [\partial_t], [u_1\partial_{u_1}], [u_2\partial_{u_1}], [u_1\partial_{u_2}] \},\\
H_{book}^2(U) &\cong &\R^3, \ \mbox{ with basis } \{[u_1\partial_t\wedge \partial_{u_1}], [u_2\partial_t\wedge \partial_{u_1}],[u_1\partial_t\wedge \partial_{u_2}]\},\\
H_{book}^3(U) &= &0,
\end{eqnarray*}
in agreement with \cite[$\S$ 2.5]{Zeiser} (when $U=\R^3$). \hfill $\diamond$
\end{example}

For an arbitrary $n\geq 2$,  it  follows from the theorem that 
\begin{itemize}
\item $H_{book}^0(U)\cong \R$ (the fact that Casimir functions are constant can be directly checked from the leaf decomposition of $\Lambda_{book}$), 
\item
$H^1_{book}(U)\cong \R^{(n-1)^2}$, with basis 
$$
\{ [\partial_t], [u_i\partial_{u_j}]\ :\ (i,j)\neq(n-1, n-1) \}.
$$
As a Lie algebra, $H^1_{book}(U)$  is isomorphic to 
$$
\R\oplus \frac{\mathfrak{gl}_{n-1}(\R)}{\R \cdot\mathrm{Id}}\cong \R\oplus\mathfrak{sl}_{n-1}(\R)\cong  \mathfrak{gl}_{n-1}(\R).
$$
\item $H^n_{book}(U)=0$, which follows from the fact that multivector fields of the form $\partial_t\wedge (u_I\partial_{u_1}\wedge\cdots\wedge\partial_{u_{n-1}})$ with $|I|=n-1$ can be written as 
$$
    \pm E\wedge \partial_t\wedge u_{(I\setminus i)}\partial_{u_1}\wedge\cdots\wedge \widehat{\partial}_{u_i}\wedge\cdots\wedge\partial_{u_{n-1}},
$$ 
for every $i\in I$.
\end{itemize}

Obtaining a basis for (and hence the dimension of $H^k_{book}(U)$, $k>1$), requires additional work, as we now explain.

Let $J = (j_1,\ldots,j_k)$ be an increasing multi-index, i.e., $1\leq j_s\leq (n-1)$ and  $j_1<\ldots<j_k$,  and denote by $J^c$ the increasing multi-index that ``complements'' $J$, in the sense that its underlying set is 
$$
\{j\in \{1,2,\ldots, n-1\}\ |\ j\notin \{j_1,\ldots,j_k\} \}.
$$
\begin{definition}
Given multi-indices $I$ and $J$, with $J$ (strictly) increasing, we say that the pair $(I,J)$ is {\em admissible} if $\mathrm{max}(I)\leq \max(J^c)$ \footnote{If $I$ or $J^c$ are empty, the definition still makes sense with the convention that $\mathrm{max}(\emptyset)=-\infty$.}. More explicitly, this means that 
\begin{itemize}
\item If $n-1\notin J$, then $(I,J)$ is admissible for any $I$ (since in this case $\max(J^c)=n-1$). 
\item If $n-1\in J$, let $r$ be the smallest element in $J$ such that $r,r+1,\ldots, n-1\in J$. Then $(I,J)$ is admissible if and only if
$I\cap \{r,r+1,\ldots,n-1\}=\emptyset$. 
\end{itemize}
\end{definition}

\begin{theorem}\label{thm:main2}
Let $U\subset \R^n$ be an open subset as in Assumption \ref{assump:U}.
Then
$$
H^k_{book}(U) \cong \R^{b_k},
$$
where $b_k = \binom{n-1}{k}\binom{n+k-2}{k}$, with basis
$$
\left\{\ [\partial_t \wedge (u_{I'}\partial_{J'})],\ [u_I \partial_J]\;: \; \begin{array}{l}
(I',J')\, \mathrm{ and }\,\, (I,J)\, \mathrm{admissible}, \\
|I'|=|J'|=k-1,\; |I|=|J|=k
\end{array}
\right\}.
$$
\end{theorem}

\section{Proofs of the main results}\label{sec:proofs}

We start by laying out the general strategy to prove Theorems \ref{thm:main1} and \ref{thm:main2}.

Consider a (nonempty) open subset $U\subseteq \R^n$, with coordinates $(t, u_1,\ldots u_{n-1})$. 

Define 
$$
\mathbb{W}^\bullet := \wedge^\bullet \mathrm{span}_\R\{ u_i\partial_{u_j}\;:\; i,j=1,\ldots, n-1\}  \subseteq \mathfrak{X}^\bullet(U), 
$$ 
so that $\mathbb{W}^k$ can be regarded as the subspace of $\mathfrak{X}^k(U)$ consisting of $k$-vector fields on $U$ with no $\partial_t$-component and coefficients given by homogeneous polynomials of degree $k$ only depending on the coordinates $(u_1,\ldots,u_{n-1})$. 
Note that $u_I\partial_J$, with $|I|=|J|=k$, form a basis for $\mathbb{W}^k$, from where it follows that 
\begin{equation}\label{eq:dimWk}
\dim(\mathbb{W}^k) =\binom{n-1}{k}\binom{n-2+k}{k}.
\end{equation}

We have a natural map 
$$
\mathbb{W}^{k-1}\stackrel{E\wedge - }{\longrightarrow} \mathbb{W}^k.
$$

The next subsections will perform the following steps:

\begin{itemize}

\item[{\bf Step 1.}] We will introduce a complex $(\mathcal{C}^\bullet_\mathcal{I}, \delta)$ whose cohomology in degree $k$ is given by
$$
\frac{\mathbb{W}^{k-1}}{E\wedge \mathbb{W}^{k-2}}\oplus \frac{\mathbb{W}^k}{E\wedge \mathbb{W}^{k-1}}.
$$
See Proposition~\ref{prop:CIcoh}.

\item[{\bf Step 2.}] For $U\subset\R^n$ satisfying Assumption \ref{assump:U}, we construct a quasi-isomorphism
$\mathcal{T}: (\mathfrak{X}^\bullet(U), d_{book}) \to (\mathcal{C}_\mathcal{I}^\bullet,\delta)$, whose inverse map in cohomology is 
\begin{equation}\label{eq:decompk}
\frac{\mathbb{W}^{k-1}}{E\wedge \mathbb{W}^{k-2}}\oplus \frac{\mathbb{W}^k}{E\wedge \mathbb{W}^{k-1}} \stackrel{\sim}{\longrightarrow} H^k_{book}(U), \; ([a],[b])\mapsto [\partial_t\wedge a + b].
\end{equation}
We deduce from this result that the natural map 
$$
\wedge^\bullet\mathrm{span}_\R\{\partial_t,u_i\partial_{u_j}\}\to H^\bullet_{book}(U)
$$
is surjective and has kernel $\langle E\rangle$, thereby proving Theorem~\ref{thm:main1}.
See Theorem~\ref{thm:poissoncohU}.

\item[{\bf Step 3.}] We prove (see Theorem~\ref{thrm:complement}) that 
$$
\mathbb{W}^k=\mathrm{span}\{u_I \partial_J, \;\; (I,J)\, \mathrm{admissible}\}\oplus E\wedge \mathbb{W}^{k-1},
$$ 
which implies Theorem~\ref{thm:main2}.

\end{itemize}

\subsection{The complex $(\mathcal{C}^\bullet_\mathcal{I}, \delta)$}
Let $\mathcal{I}\subseteq \R$ be an open interval, and consider the complex
$\mathcal{D}_\I = (\mathcal{D}^\bullet_\I, \delta_\mathcal{D})$ given by 
$$
\mathcal{D}^k_{\mathcal{I}}=C^\infty(\mathcal{I})\otimes\mathbb{W}^k, \,\mbox{ and }\;
\delta_\mathcal{D}(\alpha) =  E\wedge \partial_t\alpha.
$$ 
(Writing $\alpha= f(t)\otimes a$ with $a\in\mathbb{W}^k$, then
$E\wedge \partial_t\alpha = f'(t)\otimes ( E\wedge a)$.)

We will be interested 
in the complex 
$$
    \mathcal{C}^\bullet_{\mathcal{I}} =\mathcal{D}^{\bullet-1}_{\mathcal{I}}+\overline{\mathcal{D}}_{\mathcal{I}}^\bullet,
$$ 
where $\overline{\mathcal{D}}_{\mathcal{I}}^\bullet$ is the same complex as $\mathcal{D}_{\mathcal{I}}^\bullet$ but with $-\delta_\mathcal{D}$ as differential, i.e.,  $\mathcal{C}_\I = (\mathcal{C}^\bullet_\mathcal{I},\delta)$ is given by 
$$
\mathcal{C}^k_\mathcal{I} := C^\infty(\mathcal{I})\otimes (\mathbb{W}^{k-1}\oplus \mathbb{W}^k), \qquad  \delta(\alpha\oplus \beta) = E\wedge \partial_t \alpha \oplus (- E\wedge \partial_t \beta).
$$
Our goal in this section is to calculate the cohomology of $\mathcal{C}_\I$,
$$
H^\bullet(\mathcal{C}_\I) = H^{\bullet-1}(\mathcal{D}_\I) \oplus H^\bullet(\mathcal{D}_\I).
$$

We will use the following preliminary result. For an open subset $U\subset\R^n= \R\times\R^{n-1}= \{(t,u_1\ldots,u_{n-1})\}$, $n\geq 2$, recall the notation
$$
\mathcal{W}^k(U) = \{\mu\in\mathfrak{X}^k(U) \ : \ \iota_{dt}\mu =0 \}.
$$
The operator $E\wedge- : \mathcal{W}^{k}(U) \to \mathcal{W}^{k+1}(U)$ makes $(\mathcal{W}^\bullet(U), \ E\wedge -)$ into a complex.

\begin{lemma}\label{lem:timeDependent_flatness} For any open subset $U\subseteq \R^n$,
    $$
    H^k(\mathcal{W}^\bullet(U), \ E\wedge -) \cong \left\{ \begin{array}{cc}
        0 &  \mbox{if } k\neq n-1\\
        C^\infty\big(U\cap (\R\times\{0\})\big) &  \mbox{ if } k=n-1
    \end{array}\right. .
    $$
    For $k=n-1$, the isomorphism is given by the restriction map $f(t,u)\partial_{u_1}\wedge\cdots\wedge\partial_{u_{n-1}}\mapsto f(t,0)$.
    (We use the convention that $C^\infty(\emptyset)=0$.)
\end{lemma}

The proof of this lemma can be found in Appendix~\ref{app:eulercoh}.

Let $U=\mathcal{I}\times \R^{n-1}\subseteq \R^n$, and consider, for each $k=0,1,2,\ldots, n-1$, the Taylor-projection
$$
T^{k}\colon \mathcal{W}^ k(U)\to C^\infty(\mathcal{I})\otimes \mathbb{W}^{k}
$$
given by
\begin{equation}\label{eq: def_k_th_taylor} 
f_J\partial_J \ \mapsto \ \left(\frac{1}{k!}\sum_{i_1,\ldots,i_k=1}^{n-1}\frac{\partial^k f_J}{\partial u_{i_1}\ldots \partial u_{i_k}}(t,0)u_{i_1}\ldots u_{i_k} \ \right)\partial_J.
\end{equation}
 (Some of the relevant properties of these operators can be found in App.~\ref{app:taylor}.)


\begin{proposition}\label{prop:CIcoh}
    For any $t_0\in\mathcal{I}$, the map
    $$
        H^\bullet(\mathcal{D}_\I) \to \frac{\mathbb{W}^\bullet}{E\wedge\mathbb{W}^{\bullet-1}}, \quad [f\otimes a]\mapsto [f(t_0)a],
    $$ 
    is an isomorphism, with inverse map given by $[a]\mapsto [1\otimes a]$. 
    
    In particular, we obtain  an isomorphism
    \begin{equation}\label{eq:C_I}
        H^\bullet(\mathcal{C}_\I) \stackrel{\sim}{\to} \ \frac{\mathbb{W}^{\bullet-1}}{E\wedge\mathbb{W}^{\bullet-2}}\oplus \frac{\mathbb{W}^\bullet}{E\wedge\mathbb{W}^{\bullet-1}}, \quad [(f_1\otimes a_1, f_2\otimes a_2)] \mapsto ([f_1(t_0)a_1],[f_2(t_0)a_2]).
    \end{equation}
\end{proposition}
\begin{proof}
Consider the map  
$$
\mathrm{ev}_{t_0}\colon\D^\bullet_\I\to\frac{\mathbb{W}^\bullet}{E\wedge \mathbb{W}^{\bullet-1}},\;\; f\otimes a\mapsto [f(t_0)a].
$$
Since $\delta_\D(1\otimes a) = \partial_t(1) E\wedge a=0$, the injection $1\otimes\mathrm{id}\colon\mathbb{W}^k\to C^\infty(\I)\otimes\mathbb{W}^k$ takes values in the kernel of $\delta^k_{\D}$. Since $\mathrm{ev}_{t_0}\circ(1\otimes\mathrm{id})(a)=\mathrm{ev}_{t_0}(1\otimes a)=[a]$, the following diagram commutes:
    $$
    \xymatrix{
     \mathbb{W}^k \ar[r]^{1\otimes\mathrm{id}\ } \ar@{->>}[dr]& \ker \delta_\D^k \ar[d]^{\mathrm{ev}_{t_0}}\\
     & \frac{\mathbb{W}^k}{E\wedge\mathbb{W}^{k-1}},
    }
    $$
    hence $\mathrm{ev}_{t_0}\colon \ker \delta^k_\D\to \frac{\mathbb{W}^k}{E\wedge\mathbb{W}^{k-1}}$ is surjective. We will check that the kernel of this map is $\im \delta^{k-1}_\mathcal{D}$, i.e.,
    $$
     \ker \delta^k_\D\cap\ker (\mathrm{ev}_{t_0}) = \im \delta^{k-1}_\D,
    $$
    which yields an isomorphism $H^k(\mathcal{D}_\I)\to \frac{\mathbb{W}^k}{E\wedge\mathbb{W}^{k-1}}$ and concludes the proof of the proposition.

Since $\mathrm{ev}_{t_0}(\del_\D^{k-1}(\mu))= [E\wedge\partial_t\mu|_{t=t_0}]=0$, it is clear that
$\im \delta^{k-1}_\D\subset \ker \delta^k_\D\cap \ker (\mathrm{ev}_{t_0})$.
We will check that 
    $$
        \ker \delta^k_\D\cap\ker (\mathrm{ev}_{t_0}) \subseteq \im \delta^{k-1}_\D.
    $$
Let us regard $C^\infty(\I)\otimes\mathbb{W}^k$ as a subset of $\mathcal{W}^k(U)=\{\mu\in\mathfrak{X}^k(U):i_{dt}\mu=0\}$, where $U=\I\times\R^{n-1}\subset\R^n$. Take $\mu\in \ker \delta^k_\D\cap\ker (\mathrm{ev}_{t_0})\subset C^\infty(\I)\otimes\mathbb{W}^k\subset \mathcal{W}^k(U)$.  Then 
    $$
        E\wedge\partial_t\mu=0\ \mbox{ and }\ \mu|_{t=t_0}\in E\wedge\mathbb{W}^{k-1},
    $$
so $\partial_t\mu\in\mathcal{W}^k(U)$ is a cocycle in $(\mathcal{W}^\bullet(U), E\wedge -)$. By Lemma~\ref{lem:timeDependent_flatness},
    \begin{equation}\label{eq:B1_aplication}
        \partial_t\mu = E\wedge\xi, \ \mbox{ for some }\ \xi\in\mathcal{W}^{k-1}(U).
    \end{equation}
    (For $k=n-1$, since $\partial_t\mu$ has homogeneous coefficients of degree $n-1\geq 1$, we have that $\partial_t\mu|_{(t,0)}=0$, so the lemma still ensures the result.)
    
    Applying the operator $T^k\colon \mathcal{W}^k(U)\to C^\infty(\I)\otimes\mathbb{W}^k$ defined by \eqref{eq: def_k_th_taylor} on both sides of \eqref{eq:B1_aplication}, we see that $\partial_t\mu = E\wedge T^{k-1}\xi$ (see Lemma~\ref{lem:properties} (b)). So we may simply assume that $\xi\in C^\infty(\I)\otimes\mathbb{W}^{k-1}$. 

The map $\partial_t: C^\infty(\I)\otimes\mathbb{W}^{k-1} \to C^\infty(\I)\otimes\mathbb{W}^{k-1}$ is onto, so we can take ${\zeta}\in C^\infty(\I)\otimes\mathbb{W}^{k-1}$ such that $\partial_t(\partial_t{\zeta}) = \xi$. 
Note that 
$$
\partial_t(\mu - \delta_D(\zeta)) = \partial_t(\mu - E\wedge \partial_t\zeta)=0,
$$
so $\mu - \delta_D(\zeta)$ does not depend on $t$. Writing $\mu|_{t=t_0}=E\wedge a$, for $a\in \mathbb{W}^{k-1}$, we have 
$$
\mu - \delta_D(\zeta) = (\mu - \delta_D(\zeta))|_{t=t_0}=\mu|_{t=t_0}- E\wedge \partial_t\zeta|_{t=t_0}= E\wedge b,
$$
for $b= a- \partial_t\zeta|_{t=t_0} \in \mathbb{W}^{k-1}$. But $E\wedge b = \delta_D(tb)$, so $\mu = \delta_D(\zeta+ tb)$, showing that $\mu \in \im \delta^{k-1}_\D$.

    \end{proof}


\subsection{The quasi-isomorphism}
For an open subset $U\subseteq \mathbb{R}^n$ as in Assumption~\ref{assump:U}, note that $\mathcal{I}=p^{-1}(0)$ is an interval. In this case, we let 
$$
T^{k}\colon \mathcal{W}^ k(U)\to C^\infty(\mathcal{I})\otimes \mathbb{W}^{k}
$$ 
be the projection defined by \eqref{eq: def_k_th_taylor}, taking the $k$-th order homogeneous part of the Taylor polynomial at $0$ of the functions arising as coefficients of multivector fields.

Consider the complex $(\mathfrak{X}^\bullet(U), d_{book})$.

\begin{proposition}\label{prop:quasiisom}
Let $U\subseteq \mathbb{R}^n$ be  as in Assumption~\ref{assump:U}, and let $\mathcal{I}=p^{-1}(0)$. Then the map
\begin{equation}\label{eq:Tmap}
\mathcal{T}: \mathfrak{X}^\bullet(U)\to \mathcal{C}_\mathcal{I}^\bullet, \quad
\mathcal{T}(\partial_t\wedge a+b) = T^{\bullet-1}a\oplus T^\bullet b.
\end{equation}
is a quasi-isomorphism of complexes. 
\end{proposition}

As a first step to prove this proposition, we prove that it holds in the case where $U$ is a direct product. 
 An important technical tool for this result is the existence, for each $k=0, 1, \ldots$, of an operator 
$$
\Theta^k\colon C^\infty(U)\to C^\infty(U)
$$ 
that is a partial inverse to  $(E-k\mathrm{Id}): C^\infty(U)\to C^\infty(U)$, in the sense that 
\begin{equation}\label{eq:partialinv}
  \Theta^k\circ (E-k\mathrm{Id}) = (E-k\mathrm{Id})\circ \Theta^k = \mathrm{Id}-T^k.
\end{equation}
The operator $\Theta^k$ is explicitly defined by formula \eqref{eq: def_lambda_operator} (in App.~\ref{app:taylor}), and the previous identity is verified in Proposition \ref{prop: Theta_properties} (e).

\begin{lemma}[Product Case]\label{lemma: product_cohomology} 
    Let $U=\mathcal{I}\times U_0$, where $\mathcal{I}$ is an interval and $U_0$ is star-shaped with respect to $0\in\R^{n-1}$. Then Prop.~ \ref{prop:quasiisom} holds for $U$.
\end{lemma}
		\begin{proof}
			We will prove that the natural inclusion
			$$
            \iota\colon \mathcal{C}_\mathcal{I}^\bullet =  (C^\infty(\mathcal{I})\otimes\mathbb{W}^{\bullet-1})\oplus (C^\infty(\mathcal{I})\otimes\mathbb{W}^\bullet)\to\mathfrak{X}^\bullet(U), \ \ a\oplus b\mapsto \partial_t\wedge a + b$$ 
			is a quasi-inverse for $\mathcal{T}$. It follows from Lemma \ref{lem:PDE_cocyle} that $\iota$ is a map of complexes. It is also clear that $\mathcal{T}\circ \iota = \mathrm{Id}_{\mathcal{C}_{\mathcal{I}}}$. So we just need to find a homotopy between $\mathrm{Id}_{\mathfrak{X}^\bullet(U)}$ and $\iota\circ\mathcal{T}$.

For each $k=0,1 \ldots$, the operator $\Theta^k\colon C^\infty(U)\to C^\infty(U)$ above induces a map $\mathcal{W}^k(U)\to \mathcal{W}^k(U)$, $f_I\partial_I\mapsto \Theta^k(f_I)\partial_I$, that we keep denoting by $\Theta^k$. It follows from items (a) and (b) of Proposition \ref{prop: Theta_properties} that
\begin{equation}\label{eq:Thetakid}
\Theta^k \partial_t = \partial_t\Theta^k, \quad \mbox{ and } \; \Theta^k(E\wedge-) = E\wedge \Theta^{k-1}(-).
\end{equation}
Let
			$$
            \mathcal{H}^k\colon \mathfrak{X}^k(U)\to \mathfrak{X}^{k-1}(U), \ \ \mu=\partial_t\wedge a +b\mapsto \Theta^{k-1}(a).
            $$
            Using \eqref{eq:partialinv} and \eqref{eq:Thetakid}, we have
			\begin{eqnarray*}
				d^{k-1}_{book}\mathcal{H}^k(\mu)  + \mathcal{H}^{k+1}d^k_{book}(\mu)&=&  d^{k-1}_{book}(\Theta^{k-1}(a)) + \Theta^k(E\wedge\partial_ta+E(b)-kb)\\
                &=& \partial_t\wedge (E-(k-1)\mathrm{Id})(\Theta^{k-1}(a)) -E\wedge\partial_t\Theta^{k-1}(a) \\
                & & + E\wedge\partial_t\Theta^{k-1}(a)+ (b-T^kb) \\
                &=& \partial_t\wedge (a-T^{k-1}a) + b-T^kb \\
                &=& \mu - \iota\circ\mathcal{T}(\mu).
			\end{eqnarray*}
	\end{proof}

To extend the result to more general domains $U$, we need an additional lemma.

\begin{lemma}\label{lem:homogeneous_classes}
    If $U=\mathcal{I}\times U_0$, where $\mathcal{I}$ is an open interval and $U_0$ is an open set star-shaped with respect to $0\in\R^{n-1}$, then $\mathfrak{X}^\bullet(\R^n)\to\mathfrak{X}^\bullet(U)$ is a quasi-isomorphism. 
    If $U\subset\R^n$ is an open set with $U\cap (\R\times\{0\})\neq\emptyset$, then the restriction map $\mathfrak{X}^\bullet(\R^n)\to\mathfrak{X}^\bullet(U)$ is quasi-injective. 
\end{lemma}
\begin{proof}   
    We have the following commutative diagram of complexes,
    $$
    \xymatrix@R+0pc@C+1pc{
        \mathfrak{X}^\bullet(\R^n)\ar[d] \ar[r]& \mathfrak{X}^\bullet(\mathcal{I}\times U_0) \ar[d] \\
        \mathcal{C}^\bullet_\R \ar[r] & \mathcal{C}^\bullet_\mathcal{I},
    }
    $$
    where the horizontal maps are restrictions and the vertical maps are the maps $\mathcal{T}$ defined in \eqref{eq:Tmap}.
    By Lemma \ref{lemma: product_cohomology} both vertical maps are quasi-isomorphisms. Note that the same holds for the lower horizontal map: indeed, the restriction map $\mathcal{C}_\R^\bullet \to \mathcal{C}_\mathcal{I}^\bullet$ induces a map in cohomology fitting into the following commutative diagram:
\begin{equation}\label{eq:restquasi}
    \xymatrix@C=0in@R=.1in{
     H^\bullet(\mathcal{C}_\R)\ar[rr]\ar[ddr]& &H^\bullet(\mathcal{C}_{\mathcal{I}})\ar[ddl]\\
     &&\\
     & \frac{\mathbb{W}^{\bullet-1}}{E\wedge\mathbb{W}^{\bullet-2}}\oplus \frac{\mathbb{W}^\bullet}{E\wedge\mathbb{W}^{\bullet-1}},  &
    }
\end{equation}
where the diagonal maps are the isomorphisms \eqref{eq:C_I}.

For the second part, if $U\cap (\R\times\{0\})\neq\emptyset$, then we can choose an open set $\mathcal{I}\times U_{0}\subset U$, where $U_0$ is star-shaped with respect to $0\in\R^{n-1}$. We have a commutative diagram of complexes given by restriction maps:
    $$
    \xymatrix{
    \mathfrak{X}^\bullet(\R^n)\ar[r] \ar[rd] & \mathfrak{X}^\bullet(U) \ar[d]\\
     & \mathfrak{X}^\bullet(\mathcal{I}\times U_0)
    }
    $$
    The fact that $\mathfrak{X}^\bullet(\R^n)\to\mathfrak{X}^\bullet(U)$ is quasi-injective follows from  $\mathfrak{X}^\bullet(\R^n)\to \mathfrak{X}^\bullet(\mathcal{I}\times U_0)$ being a quasi-isomorphism. 
    \end{proof}

\begin{proof}[Proof of Prop.~ \ref{prop:quasiisom}]
    Denote by $U_0\subset\R^{n-1}$ the image of the projection $p\colon U\to\R^{n-1}$, which is star-shaped with respect to $0$. Then $U\subset\R\times U_0$ and we have the following commuting diagram of complexes
    $$
    \begin{tikzcd}
        \mathfrak{X}^\bullet(\R\times U_0) \arrow{r}{r} \arrow[swap]{d}{\mathcal{T}} & \mathfrak{X}^\bullet(U) \arrow{d}{\mathcal{T}_U} \\%
        \mathcal{C}^\bullet_\R \arrow{r}{}& \mathcal{C}^\bullet_{\mathcal{I}},
    \end{tikzcd}
    $$
 with horizontal maps given by restrictions, and vertical maps as in \eqref{eq:Tmap}.
Since $\mathcal{T}$ and the lower horizontal map are quasi-isomorphisms (see \eqref{eq:restquasi}), to prove that $\mathcal{T}_U$ is also a quasi-isomorphism it suffices to verify that the upper horizontal map $r$ is a quasi-isomorphism. By Lemma \ref{lem:homogeneous_classes}, the restriction $\mathfrak{X}^\bullet(\R^n)\to \mathfrak{X}^\bullet(\R\times U_0)$ is a quasi-isomorphism and the composition
$$
\mathfrak{X}^\bullet(\R^n)\to \mathfrak{X}^\bullet(\R\times U_0) \stackrel{r}{\to} \mathfrak{X}^\bullet(U)
$$
is quasi-injective, so $r$ must be quasi-injective. We now verify that $r$ is quasi-surjective. 

We will repeatedly use the following two facts:
\begin{itemize}
\item[(a)]  Any $\xi \in \mathcal{W}^k(U)$ can be written as $\xi=\partial_t\zeta$, for some $\zeta \in \mathcal{W}^k(U)$;
\item[(b)] If $\xi \in \mathcal{W}^k(U)$ satisfies $\partial_t\xi=0$, then it has a unique extension $\widehat{\xi}\in \mathcal{W}^k(\R\times U_0)$ satisfying $\partial_t\widehat{\xi}=0$.
\end{itemize}
Note that it suffices to verify these claims for smooth functions. The observation in (a) is a consequence of the fact that the projection $p\colon U\to\R^{n-1}$ admits a smooth global section (recall that $p$-fibers are intervals), which can be used as initial points for integration along the $t$-coordinate. The claim in (b) is also a direct consequence of the connectivity of $p$-fibers.
    
Now fix a class in $H^k_{book}(U)$ represented by 
$$
\mu = \partial_t\wedge a + b,
$$
with $a\in \mathcal{W}^{k-1}(U)$ and $b\in \mathcal{W}^k(U)$,
as in \eqref{eq:mudecomp}. 
We must find $\widetilde{a}\in \mathcal{W}^{k-1}(\R \times U_0)$, 
$\widetilde{b}\in \mathcal{W}^{k}(\R \times U_0)$ such that
$$
\partial_t \wedge a + b = r (\partial_t 
\wedge\widetilde{a} + \widetilde{b}) + d_{book}\zeta,
$$
for some $\zeta \in \mathfrak{X}^{k-1}(U)$.

Since $d_{book}\mu=0$, by Lemma~\ref{lem:PDE_cocyle} we have
    \begin{eqnarray}
        E\wedge\partial_ta &= & kb-E(b), \label{eq: d_book_mu=01}\\
        E\wedge \partial_tb & = & 0. \label{eq: d_book_mu=02}
    \end{eqnarray}
We claim that there exists $\xi\in\mathcal{W}^{k-1}(U)$ such that 
$$
\partial_tb = E\wedge \partial_t\xi.
$$
If $k\neq n-1$, by \eqref{eq: d_book_mu=02}, the claim follows from  Lemma \ref{lem:timeDependent_flatness} and the observation in (a) above.
For $k=n-1$, the claim also follows from Lemma \ref{lem:timeDependent_flatness} since, by 
 applying $\partial_t$ to both sides of the first identity in \eqref{eq: d_book_mu=01} and evaluating at $(t,0)\in U$, we see that $\partial_tb|_{(t,0)}=0$  (recall that $n\geq 2$).

 For $b_0=b-E\wedge\xi$ we have that $\partial_tb_0=0$, and therefore $b_0 = r(\widetilde{b_0})$ for  $\widetilde{b_0}\in \mathcal{W}^k(\R\times U_0)$ as in (b) above. Hence we can write
 $$
b = r(\widetilde{b_0}) + E\wedge\xi.
 $$
 
 Using the identity $b=b_0+E\wedge\xi$ in  \eqref{eq: d_book_mu=01} gives
    \begin{eqnarray*}
        E\wedge\partial_t a &=& kb_0-E(b_0) + kE\wedge\xi - E(E\wedge\xi) \\
        &=& kb_0-E(b_0) + kE\wedge\xi - E(E)\wedge\xi - E\wedge E(\xi) \\
        &=& kb_0-E(b_0) + E\wedge \big( (k-1)\xi - E(\xi) \big), 
    \end{eqnarray*}
    which can be rewritten as 
    $$kb_0-E(b_0) = E\wedge c$$
    for $c:= \partial_ta +\big(E-(k-1)\mathrm{Id}\big)\xi\in \mathcal{W}^{k-1}(U)$. Since $b_0$ is independent of $t$, it follows from the previous formula that 
    $$
    E\wedge\partial_tc=0.
    $$

We will repeat the previous arguments used to treat \eqref{eq: d_book_mu=02}  with Lemma~\ref{lem:timeDependent_flatness} and properties (a) and (b). 

\smallskip

\noindent {\bf The case $k\neq n$}. Since $(k-1) \neq (n-1)$, there exists $\gamma \in\mathcal{W}^{k-2}(U)$ such that 
    $\partial_tc =E\wedge\partial_t\gamma$.
    Consider $c_0 = c-E\wedge\gamma$. Then $\partial_tc_0=0$, and therefore $c_0 = r(\widetilde{c_0})$, for  $\widetilde{c_0}\in\mathcal{W}^{k-1}(\R\times U_0)$ as in (b). By  the definition of $c$ we have
    $$
    \partial_ta = c_0 + E\wedge\gamma - (E-(k-1)\mathrm{Id})\xi.
    $$

    Choose $\widetilde{\gamma}\in\mathcal{W}^{k-2}(U)$ and $\widetilde{\xi}\in\mathcal{W}^{k-1}(U)$ such that $\partial_t\widetilde{\gamma}=\gamma$ and $\partial_t\widetilde{\xi}=\xi$, and define 
    $$
    a_0 = a - tc_0 -E\wedge\widetilde{\gamma} +(E-(k-1)\mathrm{Id})\widetilde{\xi}.
    $$
    Then $\partial_ta_0=0$, and therefore $a_0=r(\widetilde{a_0})$ for $\widetilde{a_0}\in\mathcal{W}^{k-1}(\R\times U_0)$ as in (b). Hence we can write
$$
a = r(\widetilde{a_0} + t\widetilde{c_0}) + E\wedge\widetilde{\gamma} - (E-(k-1)\mathrm{Id})\widetilde{\xi}
$$

    Let $\zeta = \partial_t\wedge\widetilde{\widetilde{\gamma}} - \widetilde{\xi}\in\mathfrak{X}^{k-1}(U)$, where $\widetilde{\widetilde{\gamma}}\in \mathcal{W}^{k-2}(U)$ satisfies  
    $\partial_t\widetilde{\widetilde{\gamma}} = \widetilde{\gamma}$.  Then (by Lemma~\ref{lem:PDE_cocyle})
    \begin{eqnarray*}
        d_{book}(\zeta) &=& \partial_t \wedge \left( E\wedge\partial_t\widetilde{\widetilde{\gamma}} - (E-(k-1)\mathrm{Id})\widetilde{\xi} \right) + E\wedge\partial_t\widetilde{\xi} \\ 
        &=& \partial_t\wedge\left(E\wedge\widetilde{\gamma} -(E-(k-1)\mathrm{Id})\widetilde{\xi}\right) + E\wedge\xi\\
        &=& \partial_t\wedge \big( a - a_0 - tc_0 \big) + b - b_0 \\
        &=& \mu - \left( \partial_t\wedge(a_0+tc_0) +b_0 \right) \\
        &=& \mu - r\left( \partial_t\wedge(\widetilde{a_0}+t\widetilde{c_0}) +\widetilde{b_0} \right).
    \end{eqnarray*}
Note that $\widetilde\mu := \partial_t\wedge(\widetilde a_0+t\widetilde c_0)+\widetilde b_0$ is $d_{book}$-closed. Indeed,
since $E\wedge c_0=kb_0-E(b_0)$ and $p(U)=U_0$, we also have
$E\wedge\widetilde c_0=k\widetilde b_0-E(\widetilde b_0)$, and it follows from Lemma~\ref{lem:PDE_cocyle} that
\[
d_{book}(\widetilde\mu)=
\partial_t\wedge\bigl(E\wedge\widetilde c_0+
E(\widetilde b_0)-k\widetilde b_0\bigr)=0.
\]
Hence the previous identity $\mu-r(\widetilde\mu)=d_{book}(\zeta)$ proves that $[\mu]$ is in the image of the restriction map $H^k_{book}(\R\times U_0)\to H^k_{book}(U)$. So $r$ is quasi-surjective for $k\neq n$.

\smallskip

\noindent {\bf The case $k=n$}. We claim that $H^n_{book}(U)=0$, so the surjectivity of $r$ in cohomology is immediate. To check the vanishing of this cohomology, 
consider $\mu\in\mathfrak{X}^n(U)$, which has the form 
$$
\mu=f\partial_t\wedge\partial_1\wedge\cdots\wedge\partial_{n-1}
$$ 
for some smooth function $f\in C^\infty(U)$; here and below we use the shorthand notation $\partial_i = \partial_{u_i}$. By property (a) above there exist $a_i\in C^\infty(U)$, for $i=1,\ldots,n-1$, such that $\partial_ta_i=\partial_if$. Consider $\eta = \partial_t\wedge \alpha + \beta\in\mathfrak{X}^{n-1}(U)$, where 
    $$\alpha = \sum_{i=1}^{n-1}(-1)^{i+1}a_i\widehat{\partial_i} \ \ \mbox{ and }\ \ \beta = -f\partial_1\wedge\cdots\wedge\partial_{n-1}.$$
    Here $\widehat{\partial_i}$ denotes the ordered wedge product of $\partial_1,\ldots,\partial_{n-1}$ that omits the term $\partial_i$.
    Then
    \begin{eqnarray*}
        d_{book}\eta &=& \partial_t\wedge\left( E\wedge\partial_t\alpha +E(\beta) - (n-1)\beta\right) - \underbrace{E\wedge\partial_t\beta}_{=0}\\
        &=& \partial_t\wedge\left( E\wedge\sum_{i=1}^{n-1}(-1)^{i+1}\partial_if\ \widehat{\partial_i} + \left( (n-1)f - \sum_{i=1}^{n-1}u_i\partial_if  \right)\partial_1\wedge\cdots\wedge\partial_{n-1} \right) \\
        &=& \partial_t\wedge\left( \sum_{j=1}^{n-1}\sum_{i=1}^{n-1}(-1)^{i+1}u_j\partial_if\ \partial_j\wedge\widehat{\partial_i} + \left( (n-1)f - \sum_{i=1}^{n-1}u_i\partial_if  \right)\partial_1\wedge\cdots\wedge\partial_{n-1} \right)\\
        &=& (n-1)\mu,
    \end{eqnarray*}
    where we use that $\partial_j\wedge\widehat{\partial_i} = (-1)^{i+1}\partial_1\wedge\cdots\wedge\partial_{n-1}$ for $j=i$, otherwise it vanishes.
   It follows that $\mu = d_{book}\big(\eta/(n-1)\big)$ is exact, and the proof of the quasi-surjectivity of $r$ is complete.

\end{proof}

From Propositions~\ref{prop:CIcoh} and \ref{prop:quasiisom}, we conclude

\begin{theorem}\label{thm:poissoncohU}
Let $U\subset\R^n=\R\times\R^{n-1}$ be an open set such that the projection $p\colon U\to\R^{n-1}$, $(t,u)\mapsto u$ has connected fibers and its image $p(U)$ is star-shaped with respect to $0\in\R^{n-1}$. Then, for any $t_0$ such that $(t_0,0)\in U$ and $k\geq 0$, the map
$$
H^k_{book}(U)\to \frac{\mathbb{W}^{k-1}}{E\wedge \mathbb{W}^{k-2}}\oplus \frac{\mathbb{W}^k}{E\wedge \mathbb{W}^{k-1}},\quad [\partial_t \wedge a + b]\mapsto [T^{k-1}a |_{t=t_0}]\oplus [T^k b |_{t=t_0}]
$$
is well defined and an isomorphism, with inverse map given by $[a]\oplus[b]\mapsto [\partial_t\wedge a  + b]$.
\end{theorem}


Theorem~\ref{thm:main1} is a consequence of the previous theorem. Indeed, we have the following commutative diagram:
$$
\xymatrix{
 \wedge^\bullet \mathrm{span}_\R\{\partial_t, u_i\partial_{u_j}\} \ar[r]\ar[d]^\cong
& H^\bullet_{book}(U)\ar[d]^\cong\\
 \mathbb{W}^{\bullet-1}\oplus \mathbb{W}^\bullet\ar[r]& \frac{\mathbb{W}^{\bullet-1}}{E\wedge \mathbb{W}^{\bullet-2}}\oplus \frac{\mathbb{W}^\bullet}{E\wedge \mathbb{W}^{\bullet-1}},
}
$$
where the upper horizontal map is the natural ring homomorphism that maps a cocycle to its cohomology class, the lower horizontal map is the natural projection,
the left vertical isomorphism is given by
$$
\mathbb{W}^{k-1}\oplus \mathbb{W}^k \to \wedge^k \mathrm{span}_\R\{\partial_t,u_i\partial_{u_j}\}, \quad (a,b)\mapsto \partial_t\wedge a + b,
$$
and the right vertical isomorphism is the one of Theorem~\ref{thm:poissoncohU}.
Under the vertical isomorphism on the left, the kernel of the lower horizontal map corresponds to the ideal generated by $E$.

\subsection{Dimensions}\label{subsec:dimensions}

\begin{proposition} 
The wedge product $E\wedge-$ induces an exact sequence
\begin{equation}\label{eq:homogeneousExactness10}
    0\to\mathbb{W}^0\to\mathbb{W}^1\to\cdots\to\mathbb{W}^k\to\cdots\to\mathbb{W}^{n-1}\to 0.
\end{equation}
\end{proposition}
\begin{proof}
For each $k=0,1,\ldots,n-1$, we have to verify the exactness of
$$
    \mathbb{W}^{k-1}\stackrel{E\wedge-}{\longrightarrow} \mathbb{W}^k\stackrel{E\wedge-}{\longrightarrow} \mathbb{W}^{k+1}.
$$
Given $a\in\mathbb{W}^k$ satisfying $E\wedge a=0$, we need to find  $b\in\mathbb{W}^{k-1}$ such that $a=E\wedge b$. We regard $\mathbb{W}^k\subset\mathcal{W}^k(\R^n)$. Since $E\wedge a = 0$ and $a|_{(t,0)}=0$ for every $t\in\R$ when $k=n-1\geq 1$, it follows from Lemma~\ref{lem:timeDependent_flatness} that there exists some $\mu\in\mathcal{W}^{k-1}(\R^n)$ such that 
$$
    a = E\wedge\mu.
$$
Fix $t_0\in\R$ and let $b=T^{k-1}\mu|_{t=t_0} \in\mathbb{W}^{k-1}$. By applying the Taylor operator $T^k$ and evaluating at $t=t_0$ on both sides of the previous equation, we conclude that $a=E\wedge b$ (see Lemma~\ref{lem:properties} (b)).
\end{proof}

By the exactness of \eqref{eq:homogeneousExactness10}, we have that the map $E\wedge -$ induces an isomorphism $\frac{\mathbb{W}^{k-1}}{E\wedge \mathbb{W}^{k-2}}\cong E\wedge \mathbb{W}^{k-1}$, hence
    $$
\frac{\mathbb{W}^{k-1}}{E\wedge \mathbb{W}^{k-2}}\oplus \frac{\mathbb{W}^k}{E\wedge \mathbb{W}^{k-1}} \cong E\wedge \mathbb{W}^{k-1}\oplus \frac{\mathbb{W}^k}{E\wedge \mathbb{W}^{k-1}}.
$$
Therefore each choice of complement of $E\wedge \mathbb{W}^{k-1}$ in $\mathbb{W}^k$ induces an isomorphism
$$
H^k_{book}(U)\cong \mathbb{W}^k.
$$

In the next subsection we will explain how to explicitly find such a complement.

It follows from \eqref{eq:dimWk} that
\begin{equation}\label{eq:dimHk}
\dim H^k_{book}(U)=\binom{n-1}{k}\binom{n+k-2}{k}.
\end{equation}


\subsection{Basis for the book cohomology}
The main goal of this subsection is to prove Theorem~\ref{thm:main2}.
We will formulate the key results in a slightly more general framework that will be useful in  \cite{BL}.

For integers $k,l\geq 0$, we let $\mathbb{W}^{k,l}$ denote the set of $k$-multi-vectors whose coordinate functions are homogeneous polynomials of degree $l$ in $\R^{n-1}$ (with coordinates $(u_1,\ldots,u_{n-1})$), i.e.,
\begin{equation}\label{def: homogeneous_multivectors}
    \mathbb{W}^{k,l} = \mathrm{span}_\R\left\{ u_I\partial_J \ : \ |I|=l, \ |J|=k \ \right\}
\end{equation}
For $l<0$ or $k<0$, we define $\mathbb{W}^{k,l}=0$. Note that $\mathbb{W}^{k,k}=\mathbb{W}^k$ defined in $\S$ \ref{sec:proofs}.

Denote the canonical basis of $\mathbb{W}^{k,l}$ by 
$$
    \mathcal{B}^{k,l} = \left\{ u_I\partial_J : |I|=l,\ |J|=k \right\},
$$
and consider the  subset
$$
\mathcal{B}^{k,l}_{adm} = \{ u_I\partial_J \in \mathcal{B}^{k,l}\,|\, (I,J) \, \mbox{admissible}\}.
$$ 
We will refer to  $u_I\partial_J \in \mathbb{W}^{k,l}$ as admissible if $(I,J)$ is  admissible, i.e. if $\max I\leq \max J^c$ (see $\S$ \ref{sec:mainresults}). 

We will make use of the following notation for multi-indices:

\begin{itemize}
    \item Given two multi-indices $I_1$ and $I_2$ we define their union $I_1\cup I_2$ as the unique multi-index for which
    $$
        u_{I_1\cup I_2} = u_{I_1} u_{I_2}.
    $$
    So the multiplicities of common elements in $I_1$ and $I_2$ are added to form their multiplicity in $I_1\cup I_2$. 
    
    \item If $I_1$ and $I_2$ have a common element, then $\partial_{I_1\cup I_2}=0$. If $J_1$ and $J_2$ are increasing and disjoint there is a unique sign, denoted by $\mathrm{sgn}(J_1,J_2)\in\{-1,1\}$, such that
    $$
        \partial_{J_1}\wedge\partial_{J_2}= \mathrm{sgn}(J_1,J_2)\partial_{J_1\cup J_2}.
    $$
    \end{itemize}

The following result proves that the set of admissible elements in $\mathcal{B}^{k,l}$ is in bijection with the set of non-admissible ones in $\mathcal{B}^{k+1,l+1}$.

\begin{lemma}\label{lemma: good_bad}
    The following statements hold.
    \begin{enumerate}
        \item Let $u_I\partial_J$ be admissible and $m\in J^c$. Then $u_{I\cup m}\partial_{J\cup m}$ is non-admissible if and only if $m = \max J^c$.
        
        \item Let $u_I\partial_J$ be non-admissible and $m\in I\cap J$. Then $u_{I\setminus m}\partial_{J\setminus m}$ is admissible if and only if $m=\max (I\cap J)$.
        
        \item Suppose that $(k,l)\neq(n-1,0)$.
        If $\mathcal{B}^{k,l}_{adm}\neq\emptyset$, then the map 
        $$
        u_I\partial_J \mapsto  u_{I\cup m}\partial_{J\cup m}, \ \mbox{ where } \ m=\max J^c,
        $$
        is a bijection from $\mathcal{B}^{k,l}_{adm}$ to $(\mathcal{B}^{k+1,l+1} \setminus \mathcal{B}^{k+1,l+1}_{adm})$, whose inverse is given by
        $$u_I\partial_J\mapsto u_{I\setminus m}\partial_{J\setminus m}, \ \mbox{ where } \ m=\max(I\cap J).$$
        Furthermore, $\mathcal{B}^{k,l}_{adm}$ is empty if and only if so is $(\mathcal{B}^{k+1,l+1} \setminus \mathcal{B}^{k+1,l+1}_{adm})$. 
    \end{enumerate}
\end{lemma}
\begin{proof}
For (a), suppose that $u_I\partial_J$ is admissible and let $r=\max J^c +1$. Thus $\max I\leq r-1$. 

Given  $m\in J^c$ with $m<\max J^c=r-1$, we have
$$
\{r,r+1,\ldots,n-1\}\subset J\cup m \ \ \mbox{ and } \ \ r-1\notin J\cup m,
$$
so we have $\max (J\cup m)^c = r-1 (=\max J^c$). Thus 
$$
    \max(I\cup m)\leq r-1=\max(J\cup m)^c,
$$
proving that $(I\cup m, J\cup m)$ is admissible. 

Conversely, if $m=r-1$ then
$$
\{r-1,r,r+1,\ldots,n-1\} \subset J\cup m,
$$
so $\max (J\cup m)^c<r-1$ and $r-1=m\in I\cup m$. Thus 
$$
    \max (I\cup m)\geq r-1>\max (J\cup m)^c,
$$
showing that $(I\cup m, J\cup m)$ is not admissible.

For (b), let $u_I\partial_J$ be non-admissible and $m\in I\cap J$. First notice that, since $\max I>\max J^c$, we have $\max I\in J$ so $\max (I\cap J) = \max I$. Next,
$$
    \max (J\setminus m)^c = \max (J^c\cup m )= \left\{ \begin{array}{ll}
        \max J^c &\mbox{ if } m<\max J^c,  \\
        m & \mbox{ if } m >\max J^c.
    \end{array}\right.
$$
Let us analyze the two cases:
\begin{itemize}
    \item If $m\neq \max (I\cap J)$, then $m<\max I$, thus $\max (I\setminus m) = \max I$, which is strictly greater than both $\max J^c$ and $m$, therefore strictly greater than $\max(J\setminus m)^c$. It follows that $(I\setminus m, J\setminus m)$ is non-admissible.
    
    \item If $m=\max (I\cap J) =\max I >\max J^c$, then $\max(J\setminus m)^c = m =\max I\geq \max (I\setminus m)$. So $(I\setminus m, J\setminus m)$ is admissible. 
\end{itemize}

For (c) note that items (a) and (b) imply that
$$
    \mathcal{B}^{k,l}_{adm}\neq\emptyset \ \ \Leftrightarrow \ \ \mathcal{B}^{k+1,l+1}\setminus\mathcal{B}^{k+1,l+1}_{adm} \neq\emptyset.
$$
Indeed, item (b) gives ($\Leftarrow$). For ($\Rightarrow$), note that $k\neq n-1$. Otherwise, since $(k,l)\neq(n-1,0)$, we would have $k=n-1$ and $l>0$, thus $\mathcal{B}^{n-1,l}_{adm}=\emptyset$. Using that $k\neq n-1$, the result  follows from (a) because  $m=\max J^c$ is well defined for every $u_I\partial_J\in \mathbb{W}^{k,l}$. 

Suppose that $\mathcal{B}^{k,l}_{adm}\neq \emptyset$ and  let us prove the maps defined in the statement are inverses of one another. 
Given $u_I\partial_J\in \mathcal{B}^{k,l}_{adm}$,  by item (a) there is a unique $m$ such that $u_{I\cup m}\partial_{J\cup m}$ is non-admissible. On the other hand, 
for $u_I\partial_J\in \left(\mathcal{B}^{k+1,l+1}\setminus\mathcal{B}^{k+1,l+1}_{adm}\right)$, by item (b) there is a unique $m$ such that $u_{I\setminus m}\partial_{J\setminus m}$  is admissible.
The uniqueness of the elements being included or removed ensures that the maps in the statement are inverses of one another.

\end{proof}

\begin{lemma}\label{lemma: goodie-goodie}
    If $(k,l)\neq (n,1)$ then
    $$
    \mathcal{B}^{k,l}_{adm}\cup (E\wedge\mathcal{B}^{k-1,l-1}_{adm})
    $$
    is a basis for $\mathbb{W}^{k,l}$.
\end{lemma}
\begin{proof}

By hypothesis, $(k-1,l-1)\neq(n-1,0)$. In case $\mathcal{B}^{k-1,l-1}_{adm}=\emptyset$, by Lemma \ref{lemma: good_bad} (c), we have $\mathcal{B}^{k,l}=\mathcal{B}^{k,l}_{adm}$ and the result follows. 

When $\mathcal{B}^{k-1,l-1}_{adm}\neq\emptyset$, the proof consists in showing that a non-admissible element in $\mathcal{B}^{k,l}$ is, up to a sign, an element in $E\wedge\mathcal{B}^{k-1,l-1}_{adm}$ modulo an element in the span of $ \mathcal{B}^{k,l}_{adm}$. By Lemma \ref{lemma: good_bad}, for any non-admissible element $u_I\partial_J\in \mathbb{W}^{k,l}$, we have $u_{I\setminus m}\partial_{J\setminus m}$ admissible with $m=\max (I\cap J)$. Thus
    \begin{eqnarray}
        u_I\partial_J &=& u_{I'\cup m}\partial_{J'\cup m} = \mathrm{sgn}(m,J\setminus m) u_m\partial_m\wedge u_{I'}\partial_{J'}\nonumber \\ 
        &=& \mathrm{sgn}(m,J\setminus m)\left( E\wedge u_{I'}\partial_{J'} - \sum_{s\neq m} u_s\partial_s \wedge u_{I'}\partial_{J'} \right) \nonumber\\
        &=& \mathrm{sgn}(m,J\setminus m)\left(E\wedge u_{I'}\partial_{J'} - \sum_{s\notin J'\cup m} \mathrm{sgn}(s, J') \  u_{I'\cup s}\partial_{J'\cup s}\right), \nonumber 
        \end{eqnarray}
    where $I'=I\setminus m$ and $J'=J\setminus m$. By item (a) in Lemma \ref{lemma: good_bad}, $m$ is unique with the property that  $(I'\cup m,J'\cup m)$ is non-admissible. Therefore, for $s\notin J'\cup m$, we have $(I'\cup s, J'\cup s)$ admissible.

Modulo the span of $\mathcal{B}^{k,l}_{adm}$, each
$E\wedge v$, with $v\in\mathcal{B}^{k-1,l-1}_{adm}$, is a non-admissible canonical basis element with coefficient $\pm1$. By Lemma~\ref{lemma: good_bad}(c), these elements are pairwise distinct and exhaust the non-admissible basis elements. Hence the set in the statement of the lemma is also linearly independent, and therefore a basis.
    
    \end{proof}

Let $\mathbb{W}^{k,l}_{adm} = \mathrm{span}(\mathcal{B}^{k,l}_{adm})$, the subspace of $\mathbb{W}^{k,l}$ spanned by admissible basis elements. 

\begin{theorem}\label{thrm:complement}
$
\mathbb{W}^{k,l}= \mathbb{W}^{k,l}_{adm} \oplus \left(E\wedge \mathbb{W}^{k-1,l-1}\right).
$
\end{theorem}
\begin{proof}
If $\mathbb{W}^{k,l}=0$ the equality is trivial. So we may assume that $0\leq k\leq n-1$ and $l\geq 0$.

Lemma \ref{lemma: goodie-goodie} gives a direct sum decomposition 
    $$
    \mathbb{W}^{k,l}=  \mathbb{W}^{k,l}_{adm}\oplus E\wedge \mathbb{W}^{k-1,l-1}_{adm}.
    $$ 
To end the proof, we will verify that 
\begin{equation}\label{eq:EWadm}
E\wedge \mathbb{W}^{k-1,l-1}= E\wedge \mathbb{W}^{k-1,l-1}_{adm}.
\end{equation}
Since $(k-1,l-1)\neq (n,1)$, again by Lemma \ref{lemma: goodie-goodie} we have that
$$
    \mathbb{W}^{k-1,l-1}=  \mathbb{W}^{k-1,l-1}_{adm}\oplus E\wedge \mathbb{W}^{k-2,l-2}_{adm},
$$ 
which immediately implies \eqref{eq:EWadm} and concludes the proof.
\end{proof}

We now prove Theorem~\ref{thm:main2}.

\begin{proof}(Theorem~\ref{thm:main2})
Setting $k=l$ in Theorem \ref{thrm:complement}, we see that for $k=0,1,\ldots,n-1$, we have
    $$
        \mathbb{W}^k = \mathbb{W}^k_{\mathrm{adm}}\oplus (E\wedge\mathbb{W}^{k-1}).
    $$
Thus the image of $\{u_I\partial_J : (I,J) \ \mbox{is admissible and } |I|=|J|=k\}$ gives a basis for $\mathbb{W}^k/(E\wedge\mathbb{W}^{k-1})$. By Theorem \ref{thm:poissoncohU},  
$$
 \left\{\ [\partial_t \wedge (u_{I'}\partial_{J'})],\ [u_I \partial_J]\;: \; \begin{array}{l}
(I',J')\, \mathrm{ and }\,\, (I,J)\, \mathrm{admissible}, \\
|I'|=|J'|=k-1,\; |I|=|J|=k
\end{array}
\right\}.
$$
is a basis for $H^k_{book}(U)$. Its dimension is given in \eqref{eq:dimHk}.
\end{proof}


\appendix

 \section{Taylor series and the Euler vector field}\label{app:taylor}

Consider $\R^n=\R \times \R^{n-1}$ with coordinates $(t,u_1,\ldots,u_{n-1})$,  $\mathcal{I}\subseteq \R$ an open interval and $U_0\subseteq \R^{n-1}$ an open subset that is star-shaped with respect to $0$. Let $U=\mathcal{I}\times U_0$.

Throughout this section, for a smooth function $g\colon U\to\R$, we denote by $\widehat{g}\in C^\infty(U)$ the function
    $$
    \widehat{g}(t,u):= \int_0^1 g(t,su) \ ds,
    $$
    which satisfies $\widehat{g}(t,0)=g(t,0)$.

We use the notation 
    $$
    \frac{\widehat{\partial}g}{\partial u_i}(t,u) := \widehat{\frac{\partial g}{\partial u_i}}(t,u), \ \ i=1,2,...,n-1,
    $$
and let $\frac{\widehat{\partial}^{k}}{\partial u_{i_1}\cdots\ \partial u_{i_k}} = \frac{\widehat{\partial}}{\partial u_{i_1}}\ldots  \frac{\widehat{\partial}}{\partial u_{i_k}}$, for $i_1,\ldots,i_k\in\{1,\ldots,n-1\}$. By induction, we have that
    \[
        \frac{\widehat{\partial}^{k}g}{\partial u_{i_1}\cdots\ \partial u_{i_k}}(t,u) = \int_{[0,1]^k}s_1^{k-1}s_2^{k-2}\cdots s_{k-1}^1\ \frac{\partial^{k}g}{\partial u_{i_1}\cdots\ \partial u_{i_k}}(t,s_1s_2\cdots s_k u) \ ds_1 ds_2 \cdots ds_k.
    \]    
    In particular, 
    $$
    \frac{\widehat{\partial}^{k}g}{\partial u_{i_1}\cdots\partial u_{i_k}}(t,0) = \frac{1}{k!}\ \frac{\partial^{k}g}{\partial u_{i_1}\cdots\partial u_{i_k}}(t,0).
    $$
With the above notation, we can write
$$ 
g(t,u) = g(t,0) +\int_0^1\frac{d}{ds}(g(t,su))ds   = g(t,0) + \sum_j u_j \frac {\widehat{\partial} g}{\partial u_j}(t,u).
$$
 
For each $k=0,1,2,\ldots$, consider the operator $T^k\colon C^\infty(U)\to C^\infty(U)$,
    \begin{equation}\label{eq: def_l_th_taylor}
        T^k(g)(t,u) = \frac{1}{k!}\sum_{i_1,\ldots, i_k=1}^{n-1}\frac{\partial^k g}{\partial u_{i_1}\cdots\ \partial u_{i_k}}(t,0) \ u_{i_1}\cdots\ u_{i_k},
    \end{equation}
which is a homogeneous polynomial of degree $k$ in $u_1,\ldots,u_{n-1}$. We have the following formula for the Taylor series with remainder up to order $k\geq 0$:
    \begin{equation}\label{eq: taylor_exp}
        g = \sum_{l=0}^{k} T^l(g) + \sum_{i_1,\ldots,i_{k+1}=1}^{n-1}u_{i_1}\cdots u_{i_{k+1}}\frac{\widehat{\partial}^{k+1} g}{\partial u_{i_1}\cdots\ \partial u_{i_{k+1}}}.
    \end{equation}

\begin{lemma}\label{lem:properties} Setting $T^lg=0$ for $l<0$, the following properties hold:
\begin{enumerate}
    \item $\partial_tT^lg=T^l\partial_tg$.
    \item  $T^l(u_{i_1}\cdots\ u_{i_k}g) = u_{i_1}\cdots\ u_{i_k}T^{l-k}(g)$,
        for $i_1,\ldots,i_k\in\{1,\ldots,n-1\}$.
    \item $\frac{\partial}{\partial u_i}T^l(g) = T^{l-1}(\frac{\partial g}{\partial u_i})$.
\end{enumerate}
\end{lemma}

	Consider the operator $E\colon C^\infty(U)\to C^\infty(U)$,
 $$
 E(g) = \sum_{j=1}^{n-1}  u_j\frac{\partial g}{\partial u_j}.
 $$
Denote the identity operator on $C^\infty(U)$ by $\mathrm{Id}$.

 \begin{lemma}\label{lemma: hat_removing}
 For every $g\in C^\infty(U)$ and $k\geq -1$, we have
 $$
 \big( E-k \mathrm{Id} \big) \big(  u_{i_1}\cdots\ u_{i_{k+1}} \widehat{g} \big) = u_{i_1}\cdots\ u_{i_{k+1}}  g.
 $$
 In particular, for $k=0$ we have
 $$E\left( u_i\widehat{g} \right)= u_ig.$$
 \end{lemma}
 \begin{proof}
     For $k=-1$ we need to prove that for every $g\in C^\infty(U)$ we have $g=\widehat{g}+E(\widehat{g})$. This holds since
    \begin{eqnarray*}
        \Big(\widehat{g}+E(\widehat{g})\Big)(t,u) &=& \int_0^1g(t,su)\ ds + \sum_{i=1}^{n-1}u_i\frac{\partial}{\partial u_i}\int_0^1g(t,su)\ ds \\
        &=& \int_0^1\left( g(t,su)+s\sum_{i=1}^{n-1}u_i\frac{\partial g}{\partial u_i}(t,su) \right)\ ds \\
        &=& \int_0^1\frac{d}{ds}\Big( s g(t,su) \Big)\ ds \\
        &=& g(t,u).
    \end{eqnarray*}
    For $k=0$, the result follows from 
    $$
        E(u_i\widehat{g}) = E(u_i)\widehat{g}+u_iE(\widehat{g}) = u_i(\widehat{g}+E(\widehat{g})) = u_i g.
    $$
    If we inductively assume the result for $k=l-1\geq 0$, we have for $k=l$:
    \begin{eqnarray*}
        (E-l\mathrm{Id})(u_{i_1}\cdots u_{i_{l+1}}\widehat{g}) &=& E(u_{i_1})u_{i_2}\cdots u_{i_{l+1}}\widehat{g} + u_{i_1}E(u_{i_2}\cdots u_{i_{l+1}}\widehat{g}) - l u_{i_1}\cdots u_{i_{l+1}}\widehat{g}\\
        &=& u_{i_1}\cdots u_{i_{l+1}}\widehat{g} + u_{i_1}\big((l-1)u_{i_2}\cdots u_{i_{l+1}}\widehat{g}+u_{i_2}\cdots u_{i_{l+1}}g\big) \\
        & & - l u_{i_1}\cdots u_{i_{l+1}}\widehat{g} \\
        &=& u_{i_1}\cdots\ u_{i_{l+1}}g.
    \end{eqnarray*} 
 \end{proof}

 For each $k=-1,0,1,2,\ldots$, we consider the operator $\Theta^k\colon C^\infty(U)\to C^\infty(U)$, 
 \begin{eqnarray}\label{eq: def_lambda_operator}
     \Theta^kg &=& \sum_{l=0}^{k-1}\frac{1}{l-k}T^lg \ \ + \sum_{i_1,\ldots,i_{k+1}=1}^{n-1} u_{i_1}\cdots\ u_{i_{k+1}} \int_0^1\frac{\widehat{\partial}^{k+1} g}{\partial u_{i_1}\cdots \partial u_{i_{k+1}}}(t,su) \ \! ds \\
     &=& \sum_{l=0}^{k-1}\frac{1}{l-k}T^lg\  +   \int_0^1\frac{1}{s^{k+1}}\left( g-\sum_{l=0}^kT^lg \right)\!\! (t,su)\ \! ds , \nonumber
 \end{eqnarray}
 where the second identity follows from \eqref{eq: taylor_exp}. Notice that $\Theta^{-1}g = \widehat{g}$.

\begin{proposition}\label{prop: Theta_properties}
    The following properties hold for $k\geq 0$:
    \begin{enumerate}
        \item $\partial_t\Theta^kg = \Theta^k\partial_tg$;
        \item $\Theta^k(u_ig) = u_i\Theta^{k-1}g$;
        \item $\frac{\partial}{\partial u_i}\Theta^kg = \Theta^{k-1}\frac{\partial g}{\partial u_i}$;
        \item $E(\Theta^kg)=\Theta^kE(g)$; 
        \item $\Theta^k(E-k\mathrm{Id})g=(E-k\mathrm{Id})\Theta^kg = g-T^kg$.
        \item $E(g)=kg$ if and only if $g$ is a homogeneous polynomial of degree $k$ in the variables $u_1,u_2,\ldots, u_{n-1}$.
    \end{enumerate}
\end{proposition}
\begin{proof}
    Item (a) is a direct consequence of \eqref{eq: def_lambda_operator} and the analogous property for $T^l$ (Lemma~\ref{lem:properties} (a)).
    To prove (b), given $i=1,\ldots,n-1$, we have (using Lemma~\ref{lem:properties} (b))
    \begin{eqnarray*}
        \Theta^k(u_ig) &=& \sum_{l=0}^{k-1}\frac{1}{l-k}T^l(u_ig) + \int_0^1\frac{1}{s^{k+1}}\left( u_ig-\sum_{l=0}^kT^l(u_ig) \right)\! (t,su)\ \! ds\\
        &=& \sum_{l=0}^{k-1}\frac{1}{l-k}u_iT^{l-1}g + \int_0^1\frac{1}{s^{k+1}}\left( u_ig-\sum_{l=0}^ku_iT^{l-1}g \right)\! (t,su)\ \! ds \\
        &=& u_i\Theta^{k-1}g.
    \end{eqnarray*}
    For item (c) we differentiate both sides of \eqref{eq: def_lambda_operator} and use Lemma~\ref{lem:properties} (c) to obtain
    \begin{eqnarray*}
        \frac{\partial}{\partial u_i}\Theta^kg &=& \sum_{l=0}^{k-1}\frac{1}{l-k}\frac{\partial}{\partial u_i}T^lg + \int_0^1\frac{1}{s^{k+1}}\frac{\partial}{\partial u_i}\left(\left( g-\sum_{l=0}^kT^lg \right)\! (t,su)\right) \ \! ds \\
        &=& \sum_{l=0}^{k-1}\frac{1}{l-k}T^{l-1}\frac{\partial g}{\partial u_i} + \int_0^1\frac{1}{s^{k+1}} s\left( \frac{\partial g}{\partial u_i}-\sum_{l=0}^kT^{l-1}\frac{\partial g}{\partial u_i} \right)\!\! (t,su)\ \! ds \\
        &=& \Theta^{k-1}\frac{\partial g}{\partial u_i}.
    \end{eqnarray*}
    Now (d) follows from (b) and (c): 
    $$E(\Theta^kg) = \sum_{i=1}^{n-1}u_i\frac{\partial}{\partial u_i}\Theta^kg =  \sum_{i=1}^{n-1}u_i\Theta^{k-1} \frac{\partial g}{\partial u_i} = \sum_{i=1}^{n-1}\Theta^k\left(u_i\frac{\partial g}{\partial u_i}\right) = \Theta^kE(g).$$
    For (e), 
    denoting
    \[
        \theta_{i_1,\ldots,i_{k+1}}=\frac{\widehat{\partial}^{k+1} g}{\partial u_{i_1}\cdots \partial u_{i_{k+1}}},
    \]
    and recalling that $E(T^l g)= l T^lg$, we have
    \begin{eqnarray*}
        (E-k\mathrm{Id})\Theta^kg &=& (E-k\mathrm{Id}) \left(\sum_{l=0}^{k-1}\frac{1}{l-k}T^lg + \sum_{i_1,\ldots,i_{k+1}=1}^{n-1}u_{i_1}\cdots u_{i_{k+1}}\widehat{\theta}_{i_1,\ldots,i_{k+1}} \right) \\
        &=& \sum_{l=0}^{k-1}T^lg + \left( \sum_{i_1,\ldots,i_{k+1}=1}^{n-1}u_{i_1}\cdots u_{i_{k+1}}\ \theta_{i_1,\ldots,i_{k+1}}\right)\\
        &=& g-T^kg,
    \end{eqnarray*}
where we used Lemma \ref{lemma: hat_removing} in the second equality and \eqref{eq: taylor_exp} for the conclusion.

    For (f), we know that homogeneous polynomials of degree $k$ must satisfy $E(g)=kg$. On the other hand if $(E-k\mathrm{Id})g=0$, by (e) we also have 
    $$
    g-T^kg = (E-k\mathrm{Id})\Theta^kg = \Theta^k(E-k\mathrm{Id})g = 0,
    $$
    so $g=T^kg$ is a homogeneous polynomial of degree $k$.
    
\end{proof}

\section{Proof of Lemma~\ref{lem:timeDependent_flatness}}\label{app:eulercoh}

We will first prove the result for two special types of open subsets of $\R^n$, and then conclude the result for arbitrary open subsets $U$ by considering an appropriate covering.

\medskip

\noindent{\bf Step 1}.
Let us first assume that 
\begin{equation}\label{eq:empty}
U\cap (\R\times \{0\})=\emptyset.
\end{equation}

For any $c\in\mathcal{W}^k(U)$, and for each $i=1,2,\ldots,n-1$, there are unique $a_i\in\mathcal{W}^{k-1}(U)$ and $b_i\in\mathcal{W}^k(U)$ such that $i_{du_i}a_i=i_{du_i}b_i=0$ and
    $c=\partial_{u_i}\wedge a_i+b_i$. Denoting $E_i=E-u_i\partial_i$ we have $i_{du_i}E_i=0$ and
    \begin{equation}\label{eq: u_ic}
        u_ic=E\wedge a_i + (u_ib_i - E_i\wedge a_i).
    \end{equation}
    If $E\wedge c =0$, we must have $E\wedge(u_ib_i-E_i\wedge a_i)=0$; by contracting with $du_i$, we obtain $u_i(u_ib_i-E_i\wedge a_i)=0$. This implies that 
    $u_ib_i-E_i\wedge a_i=0$ (since the set where $u_i\neq 0$ is dense). 
    It follows from \eqref{eq: u_ic} that
    $u_ic = E\wedge a_i$. 
    Multiplying this last equation by $u_i$ and summing up for $i=1,2,\ldots, n-1$ gives $|u|^2c = E\wedge \sum_{i=1}^{n-1}u_ia_i$. We conclude that $c=E\wedge \xi$ for
    $$
    \xi = \frac{1}{|u|^2}\left(\sum_{i=1}^{n-1}u_ia_i\right)\in\mathcal{W}^{k-1}(U).
    $$
   This is well defined by \eqref{eq:empty}, which proves that $H^k(U)=0$ for all $k$.
We conclude that Lemma~\ref{lem:timeDependent_flatness} holds for $U$ satisfying \eqref{eq:empty}.

\medskip

\noindent{\bf Step 2}.
Let us now assume that $U$ is a cylinder of the form
\begin{equation}\label{eq:Zn}
Z^n=\mathcal{I}\times B^{n-1}(0,R),
\end{equation}
for some open interval $\mathcal{I}$ and $R>0$.
Denote by $C^\infty(\mathcal{I})_{n-1}$ the complex that has $C^\infty(\mathcal{I})$ in degree $n-1$ and zero elsewhere, with zero differential. 
The proof that Lemma~\ref{lem:timeDependent_flatness} holds in this case consists in showing that the map
\begin{equation}\label{eq:restric}
\mathcal{W}^\bullet(Z^n)\to C^\infty(\mathcal{I})_{n-1}, \quad f \partial_{u_1}\wedge\ldots\wedge \partial_{u_{n-1}}\mapsto f|_{(t,0)},
\end{equation}
is a quasi-isomorphism.

For $n\geq 2$, consider the projection $\vartheta\colon\R^n\to\R^{n-1}$, $(t,u_1,\ldots,u_{n-1})\mapsto (t,u_1,\ldots,u_{n-2})$; it takes $Z^n$ onto the cylinder 
    $$
    Z^{n-1}=\mathcal{I}\times B^{n-2}(0, R),
    $$
    which can be seen as a submanifold of $Z^n$ via the embedding $\iota\colon Z^{n-1}\to Z^n$, $(t,u_1,\ldots,u_{n-2})\mapsto (t,u_1,\ldots,u_{n-2},0)$.
    So we have a restriction map $C^\infty(Z^n)\to C^\infty(Z^{n-1})$, $f\mapsto f\circ\iota$ as well as an extension map $C^\infty(Z^{n-1})\to C^\infty(Z^n)$, $f\mapsto f\circ\vartheta$. These maps extend naturally to restrictions and extension maps between multi-vector fields in $\mathcal{W}^{\bullet}(Z^{n})$ and $\mathcal{W}^{\bullet}(Z^{n-1})$ as follows. Write $u'=(u_1,\ldots,u_{n-2})$. For a multivector field $a'$ on $Z^{n-1}$, $a'= \sum_J a'_J(t,u')\partial_J$, with $J\subseteq \{1,\ldots,n-2\}$, the extension map is
    $$
    a' \mapsto \sum_J (a'_J\circ \vartheta)  \partial_J.
    $$
    For a multivector field on $Z^n$ with no $\partial_{u_{n-1}}$-component, $a=\sum_J a_J(t,u',u_{n-1})\partial_J$, the restriction map is
    $$
    a \mapsto \sum_J a_J(t,u',0)\partial_J = \sum_J (a_J \circ \iota) \partial_J,
    $$
and for multivector fields of the form $b\wedge \partial_{u_{n-1}}$ it is $b\wedge \partial_{u_{n-1}} \mapsto 0$.

    Let $E=\sum_{i=1}^{n-1}u_i\partial_{u_i}$, regarded as a vector field on $Z^n$, and $E'= \sum_{i=1}^{n-2}u_i\partial_{u_i}$, regarded as a  vector field on $Z^{n-1}$. Consider the maps of complexes
    $$
    \mathcal{A}^k\colon(\mathcal{W}^k(Z^n), E\wedge-)\to (\mathcal{W}^{k-1}(Z^{n-1}), E'\wedge -),\;\; \mathcal{A}^k(c) = a|_{Z^{n-1}},
    $$ 
    where $c= a\wedge\partial_{u_{n-1}} +b$ is the unique decomposition of $c$ with $i_{du_{n-1}}a=0$ and $i_{du_{n-1}}b=0$,
    and 
    $$
\mathcal{B}^{k-1}\colon (\mathcal{W}^{k-1}(Z^{n-1}), E'\wedge -)\to(\mathcal{W}^k(Z^n), E\wedge-),\quad  \mathcal{B}^{k-1}(a) = a\wedge \partial_{u_{n-1}}.
    $$
    We claim that $\mathcal{A}$ and $\mathcal{B}$ are quasi-isomorphisms. Since $\mathcal{A}\circ\mathcal{B}=\mathrm{Id}_{\mathcal{W}^\bullet(Z^{n-1})}$, it suffices to prove that $\mathrm{Id}_{\mathcal{W}^\bullet(Z^n)}$ and $\mathcal{B}\circ\mathcal{A}$ are homotopic.
    
    Consider the operator $\Upsilon\colon C^\infty(Z^n)\to C^\infty(Z^n)$ 
    defined by
    $$
    \Upsilon(f)(t,u) = \int_0^1\frac{\partial f}{\partial u_{n-1}}(t,u_1,\ldots,u_{n-2},su_{n-1})\ ds.
    $$
    The map $\Upsilon$ has the following properties:
\begin{itemize}
\item $\Upsilon(u_{n-1}f)=f$.
\item $\Upsilon(u_jf)=u_j\Upsilon(f)$, for $j\neq n-1$.
\item $f - u_{n-1}\Upsilon(f)=f\circ\iota\circ\vartheta$.
\end{itemize}

We consider the extension of $\Upsilon$ to an operator on $\mathcal{W}^k(Z^n)$, still denoted by $\Upsilon$, via $f_I\partial_I\mapsto \Upsilon(f_I)\partial_I$. Let $\mathcal{H}^k\colon\mathcal{W}^k(Z^n)\to\mathcal{W}^{k-1}(Z^n)$ be defined by
    $$
    \mathcal{H}^k(c)=(-1)^{k-1}\Upsilon(a),
    $$
    where $c=a\wedge\partial_{u_{n-1}} +b$, with $i_{du_{n-1}}a=0$ and $i_{du_{n-1}}b=0$. 
By the properties of $\Upsilon$, we have
    \begin{eqnarray*}
        \Upsilon(a\wedge E) &=&  \Upsilon(a)\wedge E' + a \wedge \partial_{u_{n-1}}\\
        &=& \Upsilon(a)\wedge E + (a-u_{n-1}\Upsilon(a)) \wedge \partial_{u_{n-1}}\\
        &=& \Upsilon(a) \wedge E +  a|_{Z^{n-1}} \wedge \partial_{u_{n-1}}\\
        &=& \Upsilon(a) \wedge E + \mathcal{B}^{k-1}\circ\mathcal{A}^k(c).
    \end{eqnarray*}

        Using that $u_{n-1}c = a\wedge E + (u_{n-1}b-a\wedge E')$,
        applying $\Upsilon$, we get $c = \Upsilon(a)\wedge E + \mathcal{B}^{k-1}\circ\mathcal{A}^k(c) + b -  \Upsilon(a)\wedge E'$, so
        \begin{equation}\label{eq: c_decomposition}
             c = E\wedge\mathcal{H}^k(c) + \mathcal{B}^{k-1}\circ\mathcal{A}^k(c) + b -  \Upsilon(a)\wedge E'.
        \end{equation}
        Since $E\wedge c = (E'+u_{n-1}\partial_{u_{n-1}})\wedge (a\wedge\partial_{u_{n-1}} +b)= (E'\wedge a + (-1)^ku_{n-1}b) \wedge \partial_{u_{n-1}} + E'\wedge b$, we have
        \begin{eqnarray*}
            \mathcal{H}^{k+1}(E\wedge c) &=& (-1)^k\Upsilon(E'\wedge a + (-1)^ku_{n-1}b) \\
            &=&  b - \Upsilon(a)\wedge E'.
        \end{eqnarray*}
    Comparing with \eqref{eq: c_decomposition} gives    
    $$
    c - \mathcal{B}^{k-1}\circ\mathcal{A}^k(c) = E\wedge\mathcal{H}^k(c) + \mathcal{H}^{k+1}(E\wedge c).
    $$
This shows that the operator $\mathcal{H}: \mathcal{W}^\bullet(Z^n)\to \mathcal{W}^{\bullet -1}(Z^n)$ is a homotopy between $\mathcal{B}\circ \mathcal{A}$ and $\mathrm{Id}$, so $\mathcal{A}$ and $\mathcal{B}$ are quasi-inverses.

We inductively obtain a sequence of quasi-isomorphisms
$$
\mathcal{W}^\bullet(Z^n){\to} \mathcal{W}^{\bullet-1}(Z^{n-1}) {\to} \ldots {\to} \mathcal{W}^{\bullet-n+1}(Z^1)= C^\infty(\mathcal{I})_{n-1},
$$
whose composition is \eqref{eq:restric}. It follows that Lemma~\ref{lem:timeDependent_flatness} holds for $U=Z^n$.

    
\smallskip
    
 \noindent{\bf Step 3}.   Given an open subset $U\subseteq \R^n$, we will show, as in step 2, that the map
 \begin{equation}\label{eq:quasiiso}
\mathcal{W}^\bullet(U)\to C^\infty(U\cap (\R\times \{0\}))_{n-1}, \qquad f \partial_{u_1}\wedge\ldots\wedge \partial_{u_{n-1}}\mapsto f|_{(t,0)}
 \end{equation}
 is a quasi-isomorphism; this proves Lemma~\ref{lem:timeDependent_flatness}.

 Since any function in $C^\infty(U\cap (\R\times \{0\}))$ is the restriction of a function on $C^\infty(U)$, the map \eqref{eq:quasiiso} is surjective in cohomology (any element in $\mathcal{W}^{n-1}(U)$ is closed). It remains to check that this map is quasi-injective.

Let $\mu \in \mathcal{W}^k(U)$ be a cocycle that is mapped to $0$ under \eqref{eq:quasiiso}. We must show that $\mu$ is trivial in cohomology, i.e., there is a $\xi\in \mathcal{W}^{k-1}(U)$ such that $\mu = E\wedge \xi$.

Let $V=U\setminus (U\cap (\R\times \{0\}))$. Then $V$ is open in $U$, and we can cover $U\cap (\R\times \{0\})$
with cylinders of the form \eqref{eq:Zn}, all contained in $U$. It follows from steps 1 and 2 that we can find an open cover 
 $U=\cup_\alpha U_\alpha$ with the property that the restriction of the map \eqref{eq:quasiiso} to each $U_\alpha$ is quasi-injective. Note that $\mu|_{U_\alpha}$
 is a cocycle in $\mathcal{W}^k(U_\alpha)$ that is mapped to $0$ under the restriction of \eqref{eq:quasiiso}. It follows that, for each $\alpha$, there is a $\xi_\alpha$ such that $\mu|_{U_\alpha}=E\wedge \xi_\alpha$. Taking  a partition of unity $(\rho_\alpha)_\alpha$ subordinated to this cover, we define $\xi = \sum_\alpha \rho_\alpha \xi_\alpha \in\mathcal{W}^{k-1}(U)$ satisfying $\mu=E\wedge
     \xi$.

\bibliographystyle{plain}

\begin{thebibliography}{99}


\bibitem{BL}
Bursztyn, H., Lima, H.: Linearization and Poisson cohomology of certain Poisson structures on spheres. In final preparation.

\bibitem{CFMbook} Crainic, M., Fernandes, R., Marcut, I.,
{\em Lectures on Poisson geometry}.
Grad. Stud. Math., 217
American Mathematical Society, Providence, RI, 2021. xix+479 pp.



\bibitem{DZ}
Dufour, J.-P., Zung, N.-T., {\em  Poisson structures and their
normal forms}, Progress in Mathematics, 242, Birkhauser Boston,
2005.




\bibitem{GinzWein}
Ginzburg, V. L., Weinstein, A.,
Lie-Poisson structure on some Poisson Lie groups. {\em J. Amer. Math. Soc.} {\bf 5} (1992), 445--453.


\bibitem{GMP} Guillemin, V., Miranda, E., Pires, A. R.,
Symplectic and Poisson geometry on b-manifolds.
{\em Adv. Math.} {\bf 264} (2014), 864--896.


\bibitem{Zeiser}
Hoekstra, D., Zeiser, F.,
Poisson cohomology of 3D Lie algebras. {\em J. Geom. Phys.} {\bf 191} (2023), Paper No. 104862, 38 pp.




\bibitem{Lich} Lichnerowicz, A.,
Les variétés de Poisson et leurs algèbres de Lie associées. {\em J. Differential Geometry} {\bf 12} (1977), 253--300.


\bibitem{LuWe} Lu, J.-H., Weinstein, A.,
Poisson Lie groups, dressing transformations, and Bruhat decompositions. {\em J. Differential Geom.} {\bf 31} (1990), 501--526.

\bibitem{mot} 
Marcuț, I., Osorno Torres, B., Deformations of log-symplectic structures, {\em J. Lond. Math. Soc.} {\bf 90} (2014), 197--212.


\bibitem{MZ}
Marcuţ, I.,  Zeiser, F.,
The Poisson cohomology of $\mathfrak{sl}_2(\mathbb{R})^*$.
{\em J. Symplectic Geom.} {\bf 21} (2023), 603--652.






\bibitem{Vaisman}
Vaisman, I., Remarks on the Lichnerowicz-Poisson cohomology, {\em Annales de l’Institut Fourier} {\bf 40} (1990), 951--963.


\bibitem{Xu}
Xu, P., Poisson cohomology of regular Poisson manifolds. {\em Annales de l’Institut Fourier} {\bf 42} (1992), 967--988.





\end{thebibliography}

\end{document}